\documentclass{aims}
\usepackage{amsmath}
  \usepackage{paralist}
  \usepackage{graphics} 
  \usepackage{epsfig} 
\usepackage{graphicx}  \usepackage{epstopdf}
 \usepackage[colorlinks=true]{hyperref}
\hypersetup{urlcolor=blue, citecolor=red}

\def\currentvolume{X}
 \def\currentissue{X}
  \def\currentyear{200X}
   \def\currentmonth{XX}
    \def\ppages{X--XX}
     \def\DOI{10.3934/xx.xx.xx.xx}

\newcommand{\commentout}[1]{}
\newcommand{\Rea}{\Re e}
\newcommand{\Ima}{\Im m}
\newcommand{\C}{\mathbb{C}}
\newcommand{\R}{\mathbb{R}}
\newcommand{\N}{\mathbb{N}}
\newcommand{\Z}{\mathbb{Z}}

\newcommand{\ep}{\varepsilon}

\newcommand {\Chi} {{\bf \raise 2pt \hbox{$\chi$}} }
\newcommand {\f}   {\frac}
\newcommand {\p}   {\partial}
\newcommand{\beq}{\begin{equation}}
\newcommand{\beqa} {\begin{array}{rl}}
\newcommand{\eeq}{\end{equation}}
\newcommand{\eeqa}{\end{array}}
\newtheorem{theorem}{Theorem}[section]
\newtheorem{corollary}{Corollary}

\newtheorem{lemma}[theorem]{Lemma}
\newtheorem{proposition}{Proposition}

\theoremstyle{definition}
\newtheorem{definition}[theorem]{Definition}
\newtheorem{remark}{Remark}

\title[Time Asymptotics for  Critical Growth-Fragmentation Equations] 
      {Time Asymptotics for a Critical Case in Fragmentation and Growth-Fragmentation Equations}

\author[Marie Doumic and Miguel Escobedo]{}


 \email{marie.doumic@inria.fr}
 \email{miguel.escobedo@ehu.es}

\begin{document}
\maketitle

\centerline{\scshape Marie Doumic}
\medskip
{\footnotesize
   \centerline{Sorbonne Universit\'es, Inria, UPMC Univ Paris 06}
   \centerline{Lab. J.-L. Lions, UMR CNRS
  7598}
   \centerline{75005 Paris, France.}
} 

\medskip

\centerline{\scshape Miguel Escobedo$^*$}
\medskip
{\footnotesize
 \centerline{ Departamento de Matem\'{a}ticas}
 \centerline{ Universidad del Pa\'{\i}s Vasco UPV/EHU}
\centerline{Apartado 644, E-48080 Bilbao, Spain.}
   \centerline{\&}
\centerline{Basque Center for Applied Mathematics (BCAM)} 
\centerline{Alameda de Mazarredo 14}
\centerline{E-48009 Bilbao, Spain.}
}

\bigskip

 \centerline{(Communicated by the associate editor name)}

\begin{abstract}
Fragmentation and growth-fragmentation equations is a family of problems with varied and wide applications. This paper is devoted to the description of the long-time asymptotics of  two critical cases of these equations, when the division rate is constant and the growth rate is  linear or zero. The study of these cases may be reduced to the study of the following fragmentation equation:
$$\f{\p}{\p t} u(t,x)  +  u(t,x)=\int\limits_x^\infty k_0\left(\f{x}{y}\right)  u(t,y) dy.$$
Using the Mellin transform of the equation, we determine the long-time behavior of the solutions. Our results show in particular the strong dependence of this asymptotic behavior with respect to the initial data.
\end{abstract}

\section{Introduction}

Fragmentation and growth-fragmentation equations is a family of problems with varied and wide applications: phase transition, aerosols, polymerization processes, bacterial growth, systems with a chemostat etc. \cite{CL2,Drake_1972,Fredrickson_1967,DHKR2,SinkoStreifer}. That explains the continuing interest they meet. This paper is devoted to the description of the long  time asymptotic behavior of two critical cases of these equations that have been left open in the previous literature.

Under its general form, the linear growth-fragmentation equation may be written as follows.
\begin{equation}
\label{eq:croisfrag}
\f{\p}{\p t} u(t,x) + \f{\p}{\p x} \big(\tau (x) u (t,x)\big) + B(x)u(t,x)=\int\limits_x^\infty k(y,x) B(y) u(t,y) dy,
\end{equation}
where $u$ represents the concentration of particles of size $x$ at time $t,$ $\tau$ their growth speed, $B(x)$ the total fragmentation rate of particles of size $x,$ and $k(y,x)$ is the probability that a particle of size $y$ breaks and leaves a fragment of size $x$. For the sake of simplicity, we focus here on binary fragmentation, where each agent splits into two parts, but generalization to $k$ fragments does not present any difficulty. In the case where $\tau \equiv 0,$ Equation~\ref{eq:croisfrag} is the pure linear fragmentation equation. 

Existence, uniqueness and asymptotic behaviour have been studied and improved in many articles, let us refer to \cite{Mischler:frag} for a most recent one, which also presents an exhaustive review of the literature.  Let us indicate only the results most closely related to our work.

\begin{itemize}
\item In the case $B(x)=x^\gamma$ and $\tau(x)=x^\nu$, existence and uniqueness of a steady nonnegative profile $({\mathcal U},\lambda)$ with $\lambda>0$, ${\mathcal U}(x>0)>0,$ ${\mathcal U}\in L^1(xdx)$ and trend of $u(t,x)e^{-\lambda t}$ towards $\mathcal U$ for an appropriate weighted norm  is established in \cite{M1}, under the assumption $1+\gamma-\nu >0.$ This profile $\mathcal U$ and $\lambda $ are solutions of the eigenvalue problem

\begin{equation}
\left\{
\begin{split}
&\f{\p}{\p x} \big(\tau (x) {\mathcal U} (x)\big) + (B(x)+\lambda ){\mathcal U}(x)=\int\limits_x^\infty k(y,x) B(y) {\mathcal U}(y) dy,\,\,x\ge 0\\
&\tau {\mathcal U}(x=0)=0,\qquad {\mathcal U}(x)\ge 0, \qquad \int _0^\infty {\mathcal U}(x)dx=1.
\end{split}
\right.
\label{eq:eigen}
\end{equation}

Though this result may be generalized or refined in different directions \cite{DG, CCM}, the assumption linking $B$ and $\tau$ for $x$ vanishing or tending to infinity remains of the same kind: if $B(x)\sim x^\gamma$ and $\tau(x) \sim x^\nu$, it is necessary that $1+\gamma - \nu >0$. This may be understood as a balance between growth and division: enough growth is necessary in the neighbourhood of zero to counterbalance fragmentation, whereas enough division for large $x$ is necessary to avoid mass loss to infinity. We refer to \cite{DG} for more details; in particular, counter-examples may be given where no steady profile exist if one of the assumptions is not fulfilled.

\

\item The fragmentation equation, \emph{i.e.} when $\tau\equiv 0$,  
\begin{equation}
\label{eq:Ffrag}
\f{\p}{\p t} u(t,x)  + B(x)u(t,x)=\int\limits_x^\infty k(y,x) B(y) u(t,y) dy,
\end{equation}
was considered in \cite{EscoMischler3} and \cite{MR2650037} for a total fragmentation rate of the form 

\noindent
$B(x)=x^\gamma $ and for 
\begin{equation}
\begin{split}
&k(x, y)=\frac {1} {x}k_0\left(\frac {y} {x} \right);\,k_0 \,\,\hbox{is a non negative measure},\\
&\hbox{supp}\, k_0\subset [0, 1],\,\,\,\int \limits_0^1 z k_0(z)dz=1,\,\,\,\int \limits_0^1 k_0(z)dz >1.
\end{split}
\label{def:probak}
\end{equation} 

(i) When $\gamma >0$, it was proved in  \cite{EscoMischler3} that for initial data $u_0$ in  $L^1(xdx)$ the function $xu(t, x)$  converges to a Dirac measure, and it does so in a self-similar way. (This term is used here and in all the following in a slightly different way than in  \cite{MR2650037}, see below in Section~\ref{sec:genconvDirac} for its precise meaning.) The self-similar profiles are solutions of a particular case of Equation (\ref{eq:eigen}), with $\tau (x)=x$, so $\nu=1$, and $\lambda =1$. The condition $\gamma >0$ may then be seen as $\gamma +1-\nu>0$ again.

(ii) For  $\gamma <0$,  it was shown in \cite{MR2650037} that the behaviour is not self-similar and strongly depends  on the initial data $u_0$. The precise convergence in the sense of measures of suitable rescalings of the solutions was proved for initial data of compact support or with exponential and algebraic decay at infinity.

\end{itemize}

In this article, we investigate Equation~(\ref{eq:croisfrag}) when $\gamma =0$, $\nu =1$ and  the function $k(x, y)$ is still given by~(\ref{def:probak}). This is one critical case where $\gamma +1-\nu=0$. It is already known that under such conditions,  there is no solution to the eigenvalue problem~(\ref{eq:eigen}) in the case of homogeneous fragmentation ($k_0\equiv2$) or if $\tau(x)=cx$ and $B(x)=B$ with two constants $c,B >0$ and $c\neq B$ \cite{DG}. On the other hand, for the fragmentation equation, if $\gamma =0$ the arguments of \cite{EscoMischler3} based on a suitable scaling of the variables  break down. How can we expect to characterize the asymptotic behaviour of the population in that case? Our initial remark is that under such conditions on $\gamma $ and $\nu$ the growth fragmentation equation and the fragmentation equation are related by a very simple change of dependent variable. Suppose  that $u$ satisfies
\begin{eqnarray}
&&\f{\p}{\p t} u +  u =\int\limits_x^\infty \f{1}{y}k_0\left(\f{x}{y}\right)u(t,y)dy, \label{eq:frag}\\
&&u(0,x)=u_0(x),\label{eq:fragdata}
\end{eqnarray}
then the function 
\begin{equation}
\label{lin:uv}
v(t, x)=e^{-ct}u(t, xe^{-ct})
\end{equation}
satisfies

\begin{equation}\label{eq:transp:frag}
\f{\p}{\p t} v + \f{\p}{\p x}(c  x v) +  v =\int\limits_x^\infty \f{1}{y}k_0\left(\f{x}{y}\right) v(t,y)dy, \qquad v(t=0,x)=u_0(x).
\end{equation}
From  the  behaviour of one of them it is then easy to deduce the behaviour of the other.
 
The behaviour of Equation~(\ref{eq:frag}) has been studied  from a
probabilistic point of view in~\cite{MR2017852} where it is satisfied by the law of a stochastic fragmentation
process.   A law of large numbers and a central limit theorem were proved for
some modifications of the empirical distribution of the fragments.
As we will see below, the main novelty of our work
is that more accurate asymptotics, pointwise results and some rates of convergence are
given in Theorems~\ref{theorem1}, ~\ref{theorem2} and~\ref{theorem3} below. It also shows 
 a strong dependence from  the initial data of the asymptotic behaviour and the rates of convergence. 
Although the general behaviour of the solutions of~(\ref{eq:frag}) was essentially understood
in~\cite{MR2017852},  the present paper is a refinement and complement  obtained using different methods. 
We shall also see that it is possible to recover, at the level of the density function $u$ studied here, the
asymptotic behaviour in law of the stochastic probabilities that has been proved in \cite{MR2017852} (cf. Corollary~\ref{cor:process}).

An important quantity for the  solutions of the fragmentation equation~(\ref{eq:frag}) is the following:
\begin{equation}
\label{def:mass}
M(t)=\int _0^\infty xu(t, x)dx
\end{equation}
called sometimes the mass of the solution $u$ at time $t$. After multiplication of Equation~(\ref{eq:frag}) by $x$, integration on $(0, \infty)$ and applying Fubini's theorem, if all these operations are well defined, it follows that,
\begin{equation}
\label{prop:cons}
\frac {d} {dt}M(t)=0.
\end{equation} 
All the solutions of Equation~(\ref{eq:frag}) considered in this work satisfy that property (c.f. Theorem \ref{thm:existence}).
Since on the other hand, in the pure fragmentation equation  the particles may only fragment into smaller ones, it is natural to expect  $xu(t, x)$  to converge to a Dirac mass at the origin as $t\to \infty$. This property  is proved in Theorem~\ref{thm: toDirac} below, which states that, under suitable conditions on the initial data, $xu(t, x)$ converges to $M\delta $ in $\mathcal D'(\R^+)$ as $t\to +\infty$, where $M=\int_0^\infty x u_0(x)dx$.

The main objective of this work is to determine how this convergence takes place.  

Our initial  observation  (in Theorem~\ref{thm: noselfsim} below) is that Equation~(\ref{eq:frag}) has no solution of the form $w(t,x)=f(t)\Phi(xg(t))$ such that $\Phi \in  L^1(xdx)\cap L^\alpha (xdx)$ for some $\alpha >1$.  However it has a one parameter family of self-similar solutions that do not satisfy these conditions, see Remark~\ref{selfsim}, which are all the functions of the form
$$u_s(t,x):=x^{-s} e^{({K}(s)-1) t}=e^{({K}(s)-1) t-s\log x}.$$
Our main results actually  show that the  long-time behaviour of the solutions of  System~(\ref{eq:frag})(\ref{eq:fragdata})  strongly depend on their initial data and make appear this family of self-similar solutions. It also appears that this dependence is determined by the measure $k_0$.  This is seen by exhibiting a large set of initial data, for which the solutions to Equation~(\ref{eq:frag}) are given by the means of the Mellin transform and  where  such a dependence is seen very explicitly.

\section{Assumptions and Main Results}

\subsection{Representation formula by the means of the Mellin transform}
The solutions of Equation~(\ref{eq:frag}) may be explicitly computed for a large class of initial data by the means of the Mellin transform. Given a function $f$ defined on $(0, \infty)$, its Mellin transform is defined as follows:
\begin{equation}
\label{def:mellin}
\mathcal M _f (s)=\int \limits _{ 0 }^\infty x^{s-1}f(x)dx
\end{equation}
whenever this integral converges. In the following we denote, for the sake of simplicity
$$U(t,s):=\mathcal M _{u(t,\cdot)} (s).$$
It is easy to check that if we multiply all the terms of Equation~(\ref{eq:frag}) by $x^{s-1}$ and integrate on $(0, \infty)$, assuming that all the integrals converge and Fubini's theorem may be applied, we obtain:
\begin{equation}
\label{eq:frag:mellin}
\frac {\partial } {\partial t}U(t, s)+U(t, s)=K(s)\, U(t, s)
\end{equation}
where, by~(\ref{def:probak}), 
\begin{equation}
\label{def:mellin:noyau}
K (s)=\mathcal M _{ k_0 }(s)=\int\limits_0^1 k_0(z)z^{s-1} dz
\end{equation}
is continuous on $\Rea (s) \ge 1$, analytic in $\Rea ( s)>1$ with $K(2)=1$ and $K(1)>1$ due to~ \eqref{def:probak}. We have then, formally at least:
\begin{eqnarray}
U(t, s)&=& U_0(s)\,e^{(K(s)-1)t}\label{def:expl}\\
U_0(s)&=&\mathcal M _{u_0} (s).\label{def:expl2}
\end{eqnarray}
It only remains, in principle, to invert the Mellin transform to recover $u(t, x)$. As it is well known, in order to have an explicit formula for the inverse Mellin transform, some hypothesis on  $U_0={\mathcal  M}_{u_0}$ and $K(s)$ are needed.  If such conditions are fulfilled then,
\begin{equation}
\label{def:invmellin}
u(t, x)=\frac {1} {2\pi i}\int \limits _{ \nu-i\infty }^{ \nu+i\infty }U_0(s)\,e^{(K(s)-1)t}x^{-s}ds
\end{equation}
for some $\nu\in \R$ suitably chosen. Theorem~\ref{thm:existence} in Section~\ref{sec:explicit} below provides a rigorous setting where~\eqref{def:invmellin} holds true.  Our results on the long-time behaviour of the solutions to~(\ref{eq:frag}) are based on this explicit expression of these solutions.

\subsection{Asymptotic formula}\label{sec:main}
We are mainly interested in the asymptotic behaviour of the solution as $t\to \infty$ and how it depends on the initial data. This information will be extracted from the inverse Mellin transform in~(\ref{def:invmellin}).
In order to understand the long-time behaviour of the function $u$, it is readily seen on Equation~\eqref{def:invmellin}  that a key function is
\begin{equation}
\label{def:phi2}
\phi(s,t,x)=-s \log(x) + t {K}(s).
\end{equation}

We are then led to consider different regions  of the real half line $x>0$, determined by the different behaviour of the $x$ variable with respect to $t$.  It will be divided in several subdomains that are determined by the relative values of two parameters, that we shall denote as $p_0$, $q_0$, with respect to a third, that we call $s_+$. The two first, $p_0$ and $q_0$, only depend on the initial data $u_0$. The third parameter $s_+=s_+(t, x)$ depends on the kernel $k_0$ as well as on $t$ and $x$. In order to define these three parameters we need to precise the initial data $u_0$ and the kernel $k_0$ that we shall consider.

Our results give a somewhat  detailed  description of the long-time behaviour of the solutions $u$ of~(\ref{eq:frag})(\ref{eq:fragdata}).  This behaviour is only true for a certain set of solutions, and it  depends on several parameters of the initial data $u_0$. These must then satisfy several specific conditions. This may be seen as a drawback of our method.

The initial data $u_0$ will be assumed to be a function satisfying  the following conditions
\begin{equation}
u_0\ge 0,\,\,\,\,\int \limits _{0}^\infty u_0(x)(1+x)dx<\infty.\label{initialdata:int1}
\end{equation}
By the condition~(\ref{initialdata:int1}), the Mellin transform of the initial data $\mathcal M(0, \alpha )$ is analytic on the strip $\Rea (\alpha )\in (1,2)$ at least. Let us denote:
\begin{equation}
 I (u_0)=\left\{p\in \R;\,\,\, \int \limits _{ 0 }^\infty u_0(x)x^{p-1}dx<\infty\right\}. \label{initialdata:Ip}
 \end{equation}
 This is necessarily an interval of $\R$ and, by Hypothesis~(\ref{initialdata:int1}), $[1,2]\subset  I(u_0)$.
We then define:
 \begin{eqnarray}
&&p_0=\inf \{p; \,\,p\in I (u_0)\}\in [-\infty, 1]\label{initialdata:pzero} \\
&&q_0=\sup\{q;  \,\,q\in I (u_0)\} \in [2, +\infty]\label{initialdata:qzero}.
\end{eqnarray}

By~(\ref{initialdata:int1}) we have $p_0\le 1$ and $q_0\ge 2$. If  $q_0=+\infty$ (resp. $p_0=-\infty$), Assumptions~\eqref{initialdata:reg1}--\eqref{initialdata:reg3} (resp. \eqref{initialdata:reg1bis}--\eqref{initialdata:reg3bis}) are meaningless and useless, since Theorems~\ref{theorem1} and \ref{theorem2} (ii) (resp. Theorem~\ref{theorem2} (i)) are empty: they correspond to an empty domain of $\R_+^2$ for $(t,x)$. Next, we assume:

 \begin{eqnarray}
&&\hskip -1cm\forall \nu \in (p_0, q_0):  \int \limits _{0}^\infty |u_0'(x)|x^\nu dx<\infty,\,\,\,\lim _{ x\to 0 }x^\nu u_0(x)=\lim _{ x\to \infty}x^\nu u_0(x)=0,\label{initialdata:int2}\\
&&\hskip -1cm \forall \nu \in (p_0, q_0):  \int \limits _{0}^\infty |u_0''(x)|x^{1+\nu} dx<\infty,\,\,\,\lim _{ x\to 0 }x^{1+\nu} u_0'(x) =\lim _{ x\to \infty}x^{1+\nu} u_0'(x)=0.\label{initialdata:int3}
\end{eqnarray}
These hypothesis will be used in order to have the explicit expression~(\ref{def:invmellin}) for the solution $u$ of  System~(\ref{eq:frag}), (\ref{eq:fragdata}). 
 Some  of  our results on the asymptotic behaviour need furthermore the following conditions on the initial data.
\begin{eqnarray}
\exists a_0>0,\,\exists r>q_0; && |u_0(x)-a_0x^{-q_0}|\le C _{ 1 }x^{-r},\,\,\,\forall x>1, \label{initialdata:reg1}\\
&&|u_0'(x)+a_0q_0x^{-q_0-1}|\le C _{2 }x^{-r-1} \,\,\,\forall x>1\label{initialdata:reg2}\\
&& \hskip -1.8cm \int _1^\infty x^{\nu+1}\left|u_0''(x)-a_0q_0(q_0+1)x^{-q_0-2} \right|dx<\infty,\,\, \forall \nu\in (q_0, r).\label{initialdata:reg3}
 \end{eqnarray}
 
 \begin{eqnarray}
 \exists b_0>0,\,\exists \rho <p_0; && |u_0(x)-b_0x^{-p_0}|\le  C' _{1 }x^{-\rho },\,\,\,\forall x\in (0, 1), \label{initialdata:reg1bis}\\
&&|u_0'(x)+b_0p_0x^{-p_0-1}|\le  C' _{ 2 }x^{-\rho -1} \,\,\,\forall x\in (0, 1)\label{initialdata:reg2bis}\\
&& \hskip -1.8cm\int _0^1 x^{\nu+1}\left|u_0''(x)-b_0p_0(p_0+1)x^{-p_0-2} \right|dx<\infty,\,\, \forall \nu\in (\rho, p_0).\label{initialdata:reg3bis}
 \end{eqnarray}
for some positive constants $C_1, C_2,  C'_1, C'_2$.
 
These assumptions  impose that $u_0$ behaves
like $a_0x^{-q_0}$ as $x \to \infty$ and like $b_0x^{-p_0}$ as $x \to 0$, and give their behaviour up to second order terms $x^{-r}$ and $x^{-\rho}$.

We do not know what the results would be  under weaker conditions on $u_0$. We may notice that  in Theorems~\ref{theorem1}, \ref{theorem2} and \ref{theorem3} stated below,  the principal parts of the expansions of the solution $u$ only depend on the parameters $p_0$, $q_0$ and the function $U_0$.  The properties on the derivatives of $u_0$ only appear in the lower order and remaining terms.  It is then conceivable that some convergence of  $u$ towards these principal terms  remain true without the conditions~(\ref{initialdata:int2})--(\ref{initialdata:reg3bis}), although without decay rate estimates.

The kernel $k(x, y)$ is  defined by~(\ref{def:probak}), from where we  already saw that ${ K} (s)$ is analytic in $(1, \infty)$, continuous  and strictly decreasing on $[1,\infty).$  We complete this assumption by considering, as for $u_0$, the integral $I(k_0)$ defined by~\eqref{initialdata:Ip} and define
\begin{equation}\label{k0:p1}
p_1=\inf \{p; \,\, p\in I(k_0)\}.
\end{equation}
The relative position of $p_0$ and $p_1$ is discussed after Theorem~\ref{theorem2}. Under these conditions on $k_0$, 
it may be checked that the function $K$ is strictly convex on $(p_1,+\infty)$ and that for any $t>0$, $x\in (0, 1)$ there exists a unique $s=s_+(t, x)$ such that
\begin{equation}
\label{alphaplus20}
s_+(t, x)=(K')^{-1}\left(\f{\log x}{t}\right).
\end{equation}
(c.f. Lemma \ref{lem:phi} in the appendix).

We may now come to the main results of this work, that is the  description of the long-time behaviour of the solutions $u$ of System~(\ref{eq:frag})(\ref{eq:fragdata}) given by~(\ref{def:invmellin}). 

In the region $x>1,$ the behaviour of $u$  may be easily described. That is because  for $t>0$ and $x>1$, the function $\phi (s, t, x)$ is decreasing with respect to the real part of $s$. We have then in that region
 \begin{theorem}
\label{theorem1}
 Suppose that $u_0$ satisfies~(\ref{initialdata:int1}), (\ref{initialdata:int2})-(\ref{initialdata:reg3}). Then, for all $\delta>0$ arbitrarily small
 \begin{equation}
u(t, x)= a_0\, x^{-q_0}\, e^{({ K }(q_0)-1)t}\left(1+\mathcal O\left( e^{\left( { K }(r-\delta )-{ K}(q_0)\right)t}\right) \right) 
\end{equation}
as $t\to \infty$, uniformly for all $x\ge 1$.
\end{theorem}

As it is shown in Lemma  \ref{lem:phi}, when $x\in [0, 1]$ the behaviour of the function $\phi $, and as a consequence that of the solution $u(t, x)$, is not so simple. Let us first describe the behaviour of the solutions of System~(\ref{eq:frag})(\ref{eq:fragdata}) in the region where $x\in [0, 1]$ and the balance between $x$ and $t$ leads to a result similar to Theorem~\ref{theorem1}.
\begin{theorem}
\label{theorem2}
Suppose that $u_0$ satisfies~(\ref{initialdata:int1}), (\ref{initialdata:int2}), (\ref{initialdata:int3}), $k_0$ satisfies \eqref{def:probak}, \eqref{k0:p1}. Define the domains:
\begin{eqnarray*}
D _{ p_0 }^-=\left\{x\in (0, 1);  s _+(t, x)<p_0\right\}\,\,\,\hbox{and}\,\,\,\,D _{ q_0 }^+=\left\{x\in (0, 1);  q_0<s _+(t, x)\right\}
\end{eqnarray*}
where  $s_+$ is defined in \eqref{alphaplus20}.  
Suppose moreover that $u_0$ satisfies~(\ref{initialdata:reg1bis})-(\ref{initialdata:reg3bis}). Then, 
\begin{eqnarray*}
(i)
\hskip 0.5cm u(t, x)=b_0 x^{-p_0}e^{({ K }(p_0)-1)t}
 \left(1+\mathcal O\left( e^{-\frac {t\,K''(p_0)} {2}(s_+(t, x)-p_0)^2}\right) \right),\\
\hbox{as}\,\,\,t\to \infty, \,\,\hbox{uniformly in}\,\,\, D _{ p_0 }^-.
\end{eqnarray*}

Suppose moreover that $u_0$ satisfies~(\ref{initialdata:reg1})--(\ref{initialdata:reg3}). Then,
\begin{eqnarray*}
(ii)
\hskip 0.5cm u(t, x)=a_0 x^{-q_0}e^{({ K }(q_0)-1)t}
 \left(1+\mathcal O\left( e^{-\frac {t\,K''(q_0)} {2}(s_+(t, x)-q_0)^2}\right) \right),\\
\hbox{as}\,\,\,t\to \infty, \,\,\hbox{uniformly in }\,\,\,D_{q_0}^+.
\end{eqnarray*}
\end{theorem}
Notice that if $p_0\leq p_1,$ the point $(i)$ of Theorem~\ref{theorem2} never happens. \\

Given the results stated in Theorems~\ref{theorem1} and~\ref{theorem2}, the only region which remains to study is the region $(t,x)$ defined by $p_0 < s_+ (t,x) < q_0.$ To do so, it is necessary to distinguish a very particular type of singular discrete measures $k_0$ whose support $\Sigma$ satisfies the condition that is defined below.
\\\\
\underline{Condition~H.}
We say that  a  subset $\Sigma$ of $(0, 1)$ satisfies Condition~H if:
\begin{eqnarray*}
&&\hskip -0.8cm\exists L \in \N^* \cup \{+\infty\},\, \exists \theta \in (0, 1),\,\exists (p_\ell) _{ \ell\in \N , \; \ell\leq L }\subset \N,\, 0<p_\ell <p_{\ell +1}\, \forall \ell \in \N, \;\ell \leq L-1,\\
&&\hskip -0.8cm\Sigma=\left\{\sigma_\ell\in(0, 1);\sigma_\ell=\theta^{p_\ell }\right\}, \qquad (p_\ell)_{0\leq \ell \leq L}\text{ are setwise coprime.}
\end{eqnarray*}
Note that the assumption of a coprime and ordered sequence of integers is not a restriction (we may always order a given sequence of integers, and up to a change $\theta$ to $\theta^{gcd(p_\ell)}$, we can choose a coprime sequence). See Propositions~\ref{propositioncondH1} to \ref{propositioncondH3} in the appendix for more properties of Condition~H, which lead to define a number $v_*$ and a set $Q$ by 
\begin{equation}
\label{def:vstar}
v_*=\frac {2\pi } {\log \theta}, \qquad Q=v_* \Z.
\end{equation}

Before we state the last main theorem let us remind that by the Lebesgue decomposition of a non negative bounded measure, the measure $k_0$ may be decomposed as follows:
$$
k_0(x)=g(x)dx+d\mu +d\nu
$$
where $g\in L^1(\R^+)$, $d\mu $ is a singular continuous measure  and $d\nu$ is a singular discrete measure (cf. for example \cite{MR0188387}).
\begin{theorem}
\label{theorem3}
Suppose that $u_0$ satisfies the conditions~(\ref{initialdata:int1}), (\ref{initialdata:int2}), (\ref{initialdata:int3}) and let $\varepsilon (t)$ be any function of $t$ such that $\varepsilon (t)\to 0$ but $t\varepsilon ^2(t)\to \infty$ as $t\to \infty$. \\\\
(a) Suppose that  the measure $k_0$ has a non zero absolutely continuous part or  is a discrete singular measure, whose support  $\Sigma $ does not contain $1$ and does not satisfy Condition~H.  Then for all $\delta >0$ arbitrarily small, there exists two functions $\gamma _{ \delta  } (\varepsilon )$ and $\omega _{ \delta  } (\varepsilon )$ such that:

\begin{eqnarray}
&&u(t, x)=x^{-s_+(t, x)}e^{(K(s_+(t, x))-1)t}\times \nonumber\\
&&\hskip 3cm \left(  \frac {U_0(s_+(t, x))+\omega _{ \delta  } (\varepsilon )} {\sqrt{2\pi t K''(s_+)}}\,\Theta_1(t)+\mathcal O\left(e^{- t\gamma  _{ \delta  }(\varepsilon (t)) }\right)  \right) \label{theorem3:1}\\
&&\hskip 1.5cm \hbox{as}\,\,t\to \infty,\,\,\hbox{uniformly for}\; x\;\hbox{such that}\,\,  p_0+\delta<s _+(t, x)<q_0-\delta, \nonumber\\
&&\Theta_1(t)=\left(1+
 \mathcal O\left(\frac {e^{-\frac {t\varepsilon^2 (t)} {2}K''(q_0)}} {\sqrt {t \varepsilon^2 (t)K''(q_0)}}\right)\right),\,\,\,\hbox{as}\,\,t\to \infty\label{theorem3:2}
\end{eqnarray}
where,  $\lim _{ \varepsilon \to 0}\varepsilon  ^{-2}\gamma _{ \delta  }  (\varepsilon  )=C _{ \delta  }$ for some constant $C _{ \delta  }\in (K^{''} (q_0) /2 ,  K^{''} (p_0)/2)$. The function  $\omega _{ \delta  }(\varepsilon )$, defined in~\eqref{eq:defomega}, satisfies
\begin{eqnarray}
\label{theorem3:omega}
&&|\omega _{ \delta  } (\varepsilon )|\le C (p_0+\delta,q_0-\delta)\left( \int _{R(\varepsilon )}^\infty \left|z^{q_0-\delta-1}u_0(z)\right|dz+\right. \nonumber \\
&&\hskip 4cm \left. +\int _0^{\frac {1} {R(\varepsilon )}} \left|z^{p_0+\delta-1}u_0(z)\right|dz+\varepsilon  \log R(\varepsilon )\right),
\end{eqnarray}
where the constant $C(p_0+\delta,q_0+\delta)$ depends on the integrability properties of $u_0$ (see~\eqref{intU0}), $R$ is any function such that $R(\varepsilon )\to \infty$, $\varepsilon \log R(\varepsilon )\to 0$ as $\varepsilon \to 0$.\\\\
(b)Suppose that  the measure $k_0$  is a discrete singular measure whose support satisfies Condition~H. Then for all $\delta >0$ arbitrarily small, there exists two functions $\gamma _{ \delta  } (\varepsilon,\theta  )$ and $S_{ \delta  } (\varepsilon )$ (where $\theta$ is given by Condition~H)  such that
\begin{eqnarray}
&&u(t, x)=x^{-s_+(t, x)}e^{(K(s_+(t, x))-1)t}\left(\frac{ \sum _{ k\in \Z }U_0(s_k)
 e^{-\frac {2i\pi k } {\log \theta }\log x}}{{\sqrt{2\pi t K''(s_+)}}}+\right. \nonumber \\
&&\hskip 6.5cm   \left. +\Theta_2(t, \varepsilon )+\mathcal O\left(e^{-\gamma _{ \delta  } (\varepsilon )t}\right)  \right)\!,
 \label{theorem3:3}\\
&&\hbox{as}\,\,t\to \infty,\,\,\hbox{uniformly for}\; x\;\hbox{such that}\,\,  p_0<s _+(t, x)<q_0\nonumber
\end{eqnarray}
where $s_k=s_++i k v_*$ and  $v_*$ is defined by~\eqref{def:vstar},
$\lim _{ \varepsilon \to 0}\varepsilon  ^{-2}\gamma  (\varepsilon , \theta )=C$ for some  constant $C$ in $(K^{''} (q_0) /2 ,  K^{''} (p_0)/2)$,
\begin{eqnarray}
&&\Theta_2(t, \varepsilon )= \left(1+\mathcal O\left(\frac {e^{-\frac {\varepsilon ^2} {2}tK''(q_0)}} {\sqrt {\varepsilon ^2 tK''(q_0)}}\right)\right)S _{ \delta  }(\varepsilon )\label{theorem3:4}\\ 
&&\vert S _{ \delta  }(\varepsilon ) \vert \leq C(p_0+\delta,q_0-\delta) \biggl(
\big(\varepsilon \log R(\varepsilon)\big)L (\varepsilon)+\nonumber\\
&&\hskip 1.5cm 
+L(\varepsilon )
\left(\int _{R(\varepsilon )}^\infty \left|z^{q_0-\delta-1}u_0(z)\right|dz+\int _0^{\frac {1} {R(\varepsilon )}} \left|z^{p_0+\delta-1}u_0(z)\right|dz\right)
\label{theorem3:5}
\end{eqnarray}
where $L$, $R$ are any functions of $\varepsilon $ such that $L(\varepsilon )\to \infty$,  $R(\varepsilon )\to \infty$, 
\\
\noindent
$\varepsilon L (\varepsilon)\log R(\varepsilon )\to 0$, 
$L(\varepsilon )\left(\int  _{ R(\varepsilon ) }^\infty z^{q_0-\delta-1}u_0(z)dz+ \int  _{0}^ {\frac{1}{R(\varepsilon )} } z^{p_0+\delta-1}u_0(z)dz\right)\to 0$ and $\varepsilon L (\varepsilon)\log R(\varepsilon )\to 0$ as $\varepsilon \to 0$.
\end{theorem}

\begin{remark}[Leading terms of the solution] 

The two functions $a_0\, x^{-q_0}\, e^{({ K }(q_0)-1)t}$ and $b_0 x^{-p_0}e^{({ K }(p_0)-1)t}$ that appear in the long-time  behaviour of the solution $u$ of System~(\ref{eq:frag})(\ref{eq:fragdata}) in Theorem~\ref{theorem1}, and  in Theorem~\ref{theorem2} are self-similar solutions of the equation (cf. Remark \ref{selfsim}). \\Since $q_0>0$ and ${ K} (q_0) <1$, it follows from Theorem~\ref{theorem1} and from the point (ii) of  Theorem~\ref{theorem2} that, as $t\to \infty$, the leading terms of the solution $u(t, x)$ decays exponentially fast uniformly in the domain $x>1$ and $q_0<s_+(t, x)$.  Since $xu(t, x)$ converges to a Dirac mass at $x=0$, the long-time asymptotic behaviour of the leading term of $u(t, x)$ for $x\in (0, 1)$ is more involved. In particular in the point (i) of Theorem~\ref{theorem2} and Formulas~(\ref{theorem3:1}) and (\ref{theorem3:3}) of Theorem~\ref{theorem3} some balance exists between the power law of $x$ and the time exponential term. This is described in some detail in Sec
 tion~\ref{sec:asymptot}. Let us just say here that the  scaling law of the Dirac mass formation in $t$, $x$ variables is of exponential type in all the cases.\end{remark}

The convergence of the solution $u$ to the corresponding  self-similar solutions,
 $a_0\, x^{-q_0}\, e^{({ K }(q_0)-1)t}$ or $b_0 x^{-p_0}e^{({ K }(p_0)-1)t}$
given by  Theorem~\ref{theorem1} and Theorem~\ref{theorem2} takes place at an exponential rate, uniformly for $x$ in the domains $x>1$,  or $x\in (0, 1)$ and $s_+(t, x)<p_0-\delta $ or $s_+(t, x)>q_0+\delta $ for any $\delta >0$ arbitrarily small.
\\
On the other hand by Theorem~\ref{theorem3}, for any $\delta >0$,    
\begin{equation}
\label{asth1}
u(t, x)\Omega (t, x)-U_0(s_+(t, x))=\omega  _{ \delta  }(\varepsilon(t) )+
\mathcal O\left( e^{- t\gamma  _{ \delta  }(\varepsilon (t))}+ e^{-\frac {t\varepsilon^2 (t)} {2}K''(q_0)} \right),\,t\to \infty
\end{equation} 
uniformly on $p_0+\delta <s_+(t, x)<q_0-\delta $ where
$$
\Omega(t, x) =x^{s_+(t, x)}\frac {\sqrt{2\pi t K''(s_+(t, x)}} {e^{(K(s_+(t, x))-1)t}}.
$$
By the properties of the function $\gamma  _{ \delta  }(\varepsilon )$, we have
$$
t\gamma  _{ \delta  }(\varepsilon )=C _{ \delta  }t\varepsilon ^2(t)+o\left(t\varepsilon ^2(t) \right),\,\,\,\hbox{as}\,\,t\to \infty
$$
and the first summand in the error term of (\ref{asth1}) decays exponentially in time. For the second summand we obtain a decay 
like $e^{-Ct^{a}}$ for some constants $C>0$ and $a>0$ by imposing $t\varepsilon ^2(t)\sim t^{a}$ as $t\to \infty$.  Since $\varepsilon (t)\to 0$, we must have $a<1$.  But the decay in time of the term $\omega  _{ \delta  }(\varepsilon (t)$ has no reason to be exponentially fast and we show in the following proposition that this may be false.
\begin{proposition}
\label{proposition1}
There exists a  kernel $k_0$, and  initial data $u_0$, satisfying the hypothesis of Theorem ~\ref{theorem3},  and there exists a constant $c>0$ such that  $\Omega u(t, x)-U_0(s_+(t, x))$  tends to zero at most algebraically as $t \to \infty$ along the curve $-\log x=ct$.
\end{proposition}

\begin{remark}
The hypothesis on the measure $k_0$ stated in Theorem~\ref{theorem3} are rather restrictive. The cases where $k_0$ has an absolutely continuous part or is a singular discrete measure without $1$ as a limit point are covered, but  not the case  when the measure $k_0$ has no absolutely continuous part but has a singular continuous one  or where $1$ is a limit point.
\end{remark}
\begin{remark}
Condition~H may be seen as a generalization of the ``mitotic" fragmentation kernel, where $k_0(x)=2\delta_{x=\f{1}{2}}.$ Similar assumptions have been found in other related studies, see~\cite{BDE}, Appendix~D. Equation~\eqref{theorem3:3} can be interpreted in terms of Fourier series in $y=\log x$, and  exhibits a limit which has a $\log \theta-$ periodic part in $y.$ The Poisson summation formula applied to the function $u_0(e^y) e^{s_+ y},$ for $s_+$ fixed, leads at first order to
$$u(t,x) \sim \log \theta  \f{e^{(K(s_+)-1)t}}{\sqrt{2\pi t K''(s_+)}} \sum\limits_{n\in \Z} u_0 (\theta^n x).$$
\end{remark}

Corollary~\ref{cor:process} relates our results with those contained in~\cite{MR2017852}. It  describes the behavior of the two following scalings of $u$: 
$$r(t,y)dy=te^{2ty}u(t,e^{ty})dy$$
and
$$\tilde r(t,z)dz=r\biggl(t, y_0+\f{\sigma z}{\sqrt{t}}\biggr) \f{\sigma dz}{\sqrt{t}}, $$
where  $y_0:=K'(2)$ and $\sigma^2:=K''(2).$ These two functions correspond to the laws of some random measures
$\rho_t(dy)$ and $\tilde\rho_t(dy)$ respectively,  that are considered in \cite{MR2017852} in order  to study the fragmentation process, whose law satisfies Equation~(\ref{eq:frag}) (see Section~\ref{subsec:process} for some details).
It follows from the convergence in probability proved in Theorem 1 of \cite{MR2017852} that the random measures $\rho_t(dy)$ and $\tilde\rho_t(dy)$ converge in law towards $\delta  _{ -\mu  }$ and  ${\mathcal  N}(0,1)$ as $t\to \infty$.We show in   Corollary \ref{cor:process} how it is possible to recover this result, from Theorems~\ref{theorem1}, ~\ref{theorem2} and~\ref{theorem3} in terms  of the rescaled functions $r$ and $\tilde r$.

\begin{corollary}\label{cor:process}
Let us define, for  all $t>0$ and $y\in \R$:
\begin{equation}\label{def:r}
r(t,y):=t
e^{2ty}u(t,e^{ty}),\qquad \tilde r(t,z):=r(t,y_0+ \f{\sigma z}{\sqrt{t}})\f{\sigma}{\sqrt{t}},
\end{equation}
with $y_0:=K'(2)$ and $\sigma^2:=K''(2).$ 
Under the assumptions of Theorems~\ref{theorem1}, ~\ref{theorem2} and~\ref{theorem3} a), $r(t,\cdot) \rightharpoonup \delta_{K'(2)}U_0(2)$ and $\tilde r(t,\cdot) \rightharpoonup U_0(2)G,$ with $G(z)=\f{e^{-\f{\cdot^2}{2}}}{\sqrt{2\pi} }$, in the weak sense of measures: for any bounded continuous function  $\phi$ on $\R,$ we have
$$\int\limits_{-\infty}^{+\infty} \phi(y)r(t,y)dy \to U_0(2) \phi(K'(2))\, \hbox{and}\, \int\limits_{-\infty}^{+\infty} \phi(z)\tilde r(t,z)dz \to U_0(2) \int\limits_{-\infty}^{+\infty} \phi(z) \f{e^{-\f{z^2}{2}}}{\sqrt{2\pi} }dz,$$
with $U_0(2)=\int\limits_0^\infty x u_0(x)dx$ the initial mass. 
\end{corollary}

\begin{remark}
The case $\tau (x)=x^{1+\gamma}$ and $B(x)=x^\gamma$ may be studied following similar lines. The two equations (growth fragmentation and pure fragmentation) are not  related anymore as they were before, when $\gamma =0$.  If we take Mellin transform in the pure fragmentation equation we obtain:

$$
\f{\p}{\p t} U(t, s)=(1-K(s))U(t, s+\gamma ).
$$
A similar calculation may be done for the growth-fragmentation equation. Although these are not ordinary differential equations anymore, since they are not local with respect to the $s$ variable, they may still be explicitly solved using a Wiener Hopf type argument.
\end{remark}
It is straightforward to deduce from Theorem~\ref{theorem1}, Theorem~\ref{theorem2} and Theorem~\ref{theorem3} the asymptotic behaviour of the solution $v$ of the growth fragmentation equation. But this has to be done  in terms of $e^{ct}v(t, xe^{ct})$.  As it is shown in detail in  Section~\ref{sec:croisfrag}, if the method is unchanged, the shape of the interesting domain is   modified.

The plan of the remaining of this article is as follows. In Section~\ref{sec:explicit} we explicitly solve Equation~(\ref{eq:frag}) under suitable conditions on the initial data $u_0$ and the kernel $k_0$, using the Mellin transform. In Section~\ref{sec:genconv} we prove  that the solutions $u$ of Equation~(\ref{eq:frag}) obtained in Section~\ref{sec:explicit} are such that $xu(t, x)$ converges to a Dirac mass at the origin as $t\to \infty$. We also prove that Equation~(\ref{eq:frag}) has no self-similar solutions of the form $u(t,x)=f(t)\Phi(xg(t))$ for $\Phi \in L((1+x)dx)$.
In Section~\ref{sec:main} we prove  Theorem \ref{theorem1},  Theorem \ref{theorem2}, Theorem \ref{theorem3}, Proposition \ref{proposition1} and Corollary \ref{cor:process}. Then, we relate our results with some of those obtained in \cite{MR2017852}.  In Section~\ref{sec:asymptot}, we describe in some detail the regions of the $(x, t)$ plane where the solutions of the fragmentation equation~(\ref{eq:frag}) and the growth fragmentation equation~(\ref{eq:transp:frag}) are concentrated as $t$ increases. We also present some numerical simulations where such regions may be observed.

\section{Explicit solution to the fragmentation equation~(\ref{eq:frag}).}\label{sec:explicit}

In this section we rigorously  perform the arguments  presented in the introduction leading to the explicit formula~(\ref{def:invmellin}). 

Of course,  it is possible to obtain existence and uniqueness of suitable types of solutions to the Cauchy problem for the fragmentation equation~(\ref{eq:frag}) with kernel $k(x,y)$ given by~(\ref{def:probak}) under much weaker assumptions than we are assuming in Theorem~\ref{thm:existence}. It is easily seen for example that if $u_0$ satisfies~(\ref{initialdata:int1}), there exists a unique mild solution $u\in C([0, \infty); L^1((1+x)dx))\cap C^1((0, \infty); L^1((1+x)dx))$. More general situations  are considered in  \cite{Haas2003245}. The conditions~(\ref{initialdata:int2}) and (\ref{initialdata:int3}) are imposed in order to have the representation formula~(\ref{def:invmellin}).

The main result of this section is the following.
\begin{theorem}
\label{thm:existence}
Suppose that $u_0$ satisfies the condition~(\ref{initialdata:int1}), (\ref{initialdata:int2}) and (\ref{initialdata:int3})
Suppose that the kernel  $k$ is of the form given by~(\ref{def:probak}). Then, there exists a unique function $u$ in $C([0, \infty); L^1((1+x)dx))\cap C^1((0, \infty); L^1((1+x)dx))$ that satisfies Equation~(\ref{eq:frag}) and the condition~(\ref{eq:fragdata})  for all $t>0$ and almost every $x>0$.  This solution is given by~(\ref{def:invmellin}) for all $\nu \in (p_0, q_0)$ and  satisfies the property~(\ref{prop:cons}).
\end{theorem}

\begin{proof}
Suppose that a function $u(t, x)\in C([0, \infty); L^1((1+x)dx))$ satisfies Equation~(\ref{eq:frag}). Then,  for all $t>0$ the Mellin transform of $u$ is well defined for $\Rea (s )\in [1, 2]$ and is analytic in $\Rea (s )\in (1, 2)$. If we multiply both sides of~(\ref{eq:frag}) by $x^{s-1}$ and integrate over $(0, \infty)$ we obtain:
\begin{eqnarray*}
 \f{d}{dt} \int\limits_0^\infty u(t,x)x^{s-1} dx+ \int\limits_0^\infty u(t,x)x^{s-1} dx = \int\limits_0^\infty x^{s-1}\int\limits_x^\infty \f{1}{y}k_0\left(\f{x}{y}\right)u(t,y)dy\\
 =\int\limits_0^\infty u(t,y) \int\limits_0^y x^{s-1}k_0\left(\f{x}{y}\right)dx\frac {dy} {y}= \int\limits_0^\infty u(t,y) y^{s-1} dy \int\limits_0^1 k_0(z) z^{s-1}{dz}.
\end{eqnarray*}
In the second step we have applied Fubini's Theorem, since by hypothesis \hfill \break $ k_0(z) z^{s-1}u(t,y) y^{s-1}\in L^1((0,1)\times (0, \infty))$. This is nothing but Equation~(\ref{eq:frag:mellin}) from which we deduce~(\ref{def:expl}). Since for all $t>0$, $u(t)$ is analytic on the strip $\Rea (s )\in (1, 2)$ we will immediately deduce~(\ref{def:invmellin}), for all $\nu \in (1, 2)$ as soon as we show that
\begin{equation}
\label{cond:inv}
\int \limits _{ \nu-i\infty }^{ \nu+i\infty }\left|U_0(s)\,e^{(K(s)-1)t}x^{-s}\right|ds<\infty,
\end{equation}
To this end, we notice that, by definition of $K (s)$, $K(s)\le K(1),\,\,\,\forall s\in [1, 2].$ Therefore:
\begin{eqnarray*}
\left|U_0(s)\,e^{(K(s)-1)t}x^{-s}\right|\le \left|U_0(s)x^{-\nu}\right|e^{(K(1)-1)t}
\end{eqnarray*}
Let us check that, under the assumptions made on $u_0,$ we actually have
\begin{equation}
\label{cond:inv2}
\int \limits _{ -\infty }^{\infty}|U_0(\nu+iv)|dv<\infty
\end{equation}
for all $\nu \in (p_0, q_0)$. To this end we write:
\begin{eqnarray*}
U_0(\nu+iv)=\int _0^\infty u_0(x)x^{\nu-1}x^{iv}dx=-\frac {1} {1+iv}\int _0^\infty (u_0(x)x^{\nu-1})'x^{iv+1}dx\\
=-\frac {\nu-1} {1+iv}\int _0^\infty u_0(x)x^{\nu-1}x^{iv}dx-\frac {1} {1+iv}\int _0^\infty (u_0)'(x)x^{\nu}x^{iv}dx\\
=-\frac {\nu-1} {1+iv}U_0(\nu+iv)-\frac {1} {1+iv}\int _0^\infty (u_0)'(x)x^{\nu}x^{iv}dx,
\end{eqnarray*}
where we have used~(\ref{initialdata:int2}). 
Therefore,
$$
U_0(\nu+iv)=-\frac {1} {\nu+iv}\int _0^\infty (u_0)'(x)x^{\nu+iv}dx.
$$
By the same argument, using now~(\ref{initialdata:int3}):
$$
\lim _{ x\to 0 }(u_0)'(x) x^{\nu+1}=\lim _{ x\to \infty }(u_0)'(x) x^{\nu+1}=0,
$$
we obtain:
$$
U_0(\nu+iv)=\frac {1} {(\nu+iv)(1+\nu+iv)}\int _0^\infty (u_0)''(x)x^{1+\nu+iv}dx.
$$
We deduce
\begin{equation}
\label{intMellinu}
|U_0(\nu+iv)|\le\frac {1} {\sqrt{(\nu^2+v^2)((1+\nu)^2+v^2)}}\int _0^\infty |(u_0)''(x)|x^{1+\nu}dx
\end{equation}
and by~(\ref{initialdata:int3}),  (\ref{cond:inv2}) follows.

On the other hand, it is easy to check that if $u_0$ satisfies~(\ref{initialdata:int1})-(\ref{initialdata:int3}), then the function $u$ defined by~(\ref{def:invmellin}) satisfies
$$u\in C([0, \infty); L^1((1+x)dx))\cap C^1((0, \infty); L^1((1+x)dx))$$ and solves System~(\ref{eq:frag})(\ref{eq:fragdata}) for all $t>0$ and almost every $x>0$.
\end{proof}

\section{General convergence results}
\label{sec:genconv}
\subsection{Convergence to a Dirac Mass.}
\label{sec:genconvDirac}

\begin{theorem}
\label{thm: toDirac}
Suppose that, for some $\varepsilon >0,$
$$u\in C([0, \infty); L^1((1+x^{1+\varepsilon })dx))\cap C^1((0, \infty); L^1((1+x^{1+\varepsilon })dx))$$  is a solution of Equation~(\ref{eq:frag}) with $k_0$ satisfying~\eqref{def:probak}. 
Then:
$$
\forall \varphi \in C_0(\R^+),\,\,\,\lim _{ t\to \infty }\int \limits_0^\infty x u(t, x)\varphi (x)dx=M\varphi (0),
$$
where
$$
M=\int _0^\infty x u_0(x)dx.
$$
\end{theorem}
\begin{proof}
 Multiply Equation~(\ref{eq:frag}) by $x$ and integrate between $z$ and $\infty$ where $z\ge 0$:
\begin{eqnarray}
\frac {d} {dt}\int\limits_z^\infty x u(t, x)dx+ \int\limits_z^\infty x u(t, x)dx =\int\limits_z^\infty x\int\limits_x^\infty \f{1}{y}k_0\left(\f{x}{y}\right)u(t,y)dydx \nonumber \\
=\int\limits_z^\infty  y u(t,y)  \int\limits_{z/y}^1  r k_0\left(r\right)dr dy. \label{S3E1thm: toDirac}
\end{eqnarray}
We notice that, if $z=0$,
$$
\frac {d} {dt}\int\limits_0^\infty x u(t, x)dx=0
$$
and the initial value $M$ of $\int\limits_0^\infty x u(t, x)dx$ is preserved for all time $t>0$.

Denoting now $F(x)=\int\limits_0^x y k_0 (y) dy,$ we may see $ y k_0(y)$ as a probability distribution on $(0,1)$, and $F$ as its cumulative distribution function. The inequality~(\ref{S3E1thm: toDirac}) may be written as
\begin{eqnarray*}
\frac {d} {dt}\int\limits_z^\infty x u(t, x)dx :=\f{d}{dt} H(z;t)=-\int\limits_z^\infty  F\left(\f{z}{x}\right) x u(t,x)dx:=D(z;t)\leq 0.
\end{eqnarray*}
The functions $H(z; \cdot)$ and $D(z; \cdot)$ may then play the role of the entropy and entropy dissipation functions for Equation~(\ref{eq:frag}).

If we consider now any sequence $(t_n) _{ n\in \N }$ such that $t_n\to \infty$ and define $u_n(t, x)=u(t+t_n, x)$, the sequence $(u_n)_{ n\in \N }$ is bounded in $C^1([0, \infty);\; \mathcal M_{1+\varepsilon })$. A classical argument shows the existence of a subsequence, still denoted as  $(u_n)$, and a non negative measure valued function $U(t, dx)\in C([0, \infty);\; \mathcal M_{1+\varepsilon })$ such that
\begin{eqnarray*}
x\, u_n\rightharpoonup U\,\,\,\hbox{in}\,\,\, (C_c([0, T]\times [0, +\infty )))'\\
x\, u_n(t) \rightharpoonup  U(t)\,\,\,\hbox{in}\,\,\,  (C( [0, +\infty )))'\\
\int _0^\infty U(t, dx)=M,\,\,\forall t>0
\end{eqnarray*}
and
$$
D(U(t))=-\int\limits_z^\infty  F\left(\f{z}{x}\right) U(t, dx)=0,\,\,\,\forall t>0,\,\,\forall z>0.
$$
As soon as $k_0$ is not equal to $\delta_1(x)$,  the Dirac mass at one (excluded by Assumption~\eqref{def:probak}), there is a neighbourhood $(1-\ep,1)$ of $y=1$ around which $F(y)$ is  bounded from below by a strictly positive constant. This implies that for all $z>0,$ the integral 
$\int_{[z, \f{z}{1-\ep}]} U (t, dx)=0,$ which is sufficient to ensure that $U(t, dx)= M\delta$ for all $t>0$.
\end{proof}

Theorem \ref{thm: toDirac} does not give any information about how the convergence takes place.  When the fragmentation rate is $B(x)=x^\gamma $, with $\gamma >0$, it was proved in \cite{EscoMischler3} that  the convergence to a Dirac mass is self-similar. First it was proved  that for any $M>0$ there is a unique  non negative integrable solution $g\in L^1(xdx)$ to:

\begin{equation}
\left\{
\begin{split}
&g(z )+ \f{\p}{\p z } \big(z  g (z)\big) + B(z )g(z)=\gamma \int\limits_z ^\infty k(\rho  , z) B(\rho ) g(\rho ) d\rho\\
&\int _0^\infty xg(x)dx=M.
\end{split}
\right.
\label{eq:selfsimprofile}
\end{equation}
The first equation in~(\ref{eq:selfsimprofile})  is a particular case of Equation~(\ref{eq:eigen}), with $\tau (x)=x$, so $\nu=1$, and $\lambda =1$. 

Moreover, the solution $u$ of the fragmentation equation with initial data satisfying $\int_0^\infty x u_0(x)dx=M$ is such that
$$
\lim _{ t\to \infty }\int \limits _0^\infty\left|u(t, y)-t^{\frac {2} {\gamma }}g\left( t^{\frac {1} {\gamma }} y \right)\right| ydy=0.
$$
The function $t^{\frac {2} {\gamma }}g\left( t^{\frac {1} {\gamma }} y \right)$ is a solution of the fragmentation equation, called self-similar due to its invariance by the natural scaling associated to the equation. We show in the next section that there is no such type of solutions to Equation~(\ref{eq:frag}).

When $\gamma <0$ it is well known that the mass $M(t)=\int x u(t,x) dx$ of  the solutions of the fragmentation equation~(\ref{eq:frag})  is not constant anymore but, on the contrary, is a decreasing function of time. The asymptotic behaviour of the solutions of Equation~(\ref{eq:frag}) for $\gamma  <0$ has been described in  \cite{MR2650037}, and it was proved that it is not self-similar. It was also shown that such asymptotic behaviour strongly depends on the initial data $u_0$. For example, if $u_0$ has compact support and the measure $k_0$ satisfies:
$$
\int _0^{1-z}xk_0(x)dx\sim z^{-\beta },\,\,\,\hbox{as}\,\, z\to 0
$$
for some $\beta \in [0, 1)$, then the measure $\mu (t,x)$, solution of the fragmentation equation~\eqref{eq:frag}  satisfies, for some positive constant $C _{ \alpha , \beta  }$ only depending on $\alpha $ and $\beta $:
$$
\frac {1} {M(t)}\frac {1} {C _{ \alpha , \beta  }t^{\frac {\beta } {(1-\beta )|\alpha |}}}\mu \left(t, \frac {dx} 
{C _{ \alpha , \beta  }t^{\frac {\beta } {(1-\beta )|\alpha |}}}\right)\rightharpoonup
\mu  _{ \infty }\left(dx \right), \,\,\,\hbox{as}\,\,t\to \infty
$$
in the weak sense of measures, where $x \mu_\infty (dx)$ is a probability distribution on $(0,\infty)$ characterized   by its moments,  see Theorem~1.3. in~\cite{MR2650037}, and where
$$
M(t)\sim \exp\left(-\frac {1-\beta } {|\alpha |}C _{ \alpha , \beta  }^{|\alpha |}|\alpha | t^{\frac {\beta } {(1-\beta )|\alpha |}} \right),\,\,\,\hbox{as}\,\,\,t\to \infty.
$$

\subsection{No Self similar solutions.}
The aim of this section is to prove the following result.
\begin{theorem}
\label{thm: noselfsim} Suppose that the kernel $k$ is of the form given by~(\ref{def:probak}). Then, Equation~(\ref{eq:frag}) has no solution of the form
\begin{equation}
\label{def:selfsim}
\omega (t,x)=f(t)\Phi(xg(t))
\end{equation}
where $f$ and $g$ are continuously derivable  functions from $\R^+$ to itself and $\Phi $ in $L^1((1+x)dx)$.
\end{theorem}

\begin{proof}
We denote  $\Phi(s) = \int_0^\infty x^{s-1} \Phi(x)dx$, that by hypothesis is well defined for $s\in [1, 2]$. The Mellin transform of $u(t)$ takes now the form:
$${\mathcal  M}_u(t, s)= \big(g(t)\big)^{-s} f(t) \Phi (s),$$
and Equation~(\ref{eq:frag}) gives
$$\f{\partial  {\mathcal  M}_u}{\partial t}(t, s)= \f{f'(t)}{f(t)} {\mathcal  M}_u (t, s) -s \f{g'(t)}{g(t)} {\mathcal  M}_u(t, s)=(-1+{ K}(s)) {\mathcal  M}_u(t, s)$$
Therefore:
$$(-1+{ K}(s))=\f{f'(t)}{f(t)} - s \f{g'(t)}{g(t)}$$
and 
$${K}'(s) = -\f{g'(t)}{g(t)}.$$
We deduce that ${K}'(s)$ and $\f{g'(t)}{g(t)}$ must be constants. Therefore,
$g(t)=Ae^{-at}$, for some constants $A$ and $a$,  and  ${K}(s)=a(s-1) + 2.$ Since ${K}(s) \geq 0$ for every $s \geq 1,$ we must have $a\geq 0$. Since on the other hand,  $K(s)$ is decreasing, we deduce $a\leq 0,$ and therefore $a=0.$
But that is impossible because  ${K}(1)>1$ and ${K}(2)=1$.
\end{proof}

\begin{remark}
\label{selfsim}
All the functions of the form
$$u_s(t,x):=x^{-s} e^{({K}(s)-1) t}=e^{({K}(s)-1) t-s\log x}$$
are pointwise solutions of Equation~(\ref{eq:frag}) for all $t>0$ and $x>0$ since
$$
\frac {\partial u _{ s  }} {\partial t}=({K}(s)-1)x^{-s} e^{({K}(s)-1) t},
$$
\begin{eqnarray*}
u_s(t, x)-\int _x^\infty \frac {dy} {y}k_0\left(\frac {x} {y} \right)u_s(t, y)
=x^{-s} e^{({ K}(s)-1) t}- e^{({K}(s)-1) t}\int _x^\infty \frac {dy} {y}k_0\left(\frac {x} {y} \right) y^{-s }\\
=x^{-s} e^{({K}(s)-1) t}- e^{({K}(s)-1) t}x^{-s }\int _0^1 k_0(z)z^{s }dz=x^{-s} e^{({K}(s)-1) t}(1-K (s)).
\end{eqnarray*}
Such solutions are invariant by the change of variables: 
$$
u_{\lambda }(t, x)=u\left(t+ \frac {s \log (\lambda )} {K(s)}, \lambda x\right)
$$
since:
\begin{eqnarray*}
u_{s, \lambda }(t, x)=u_s\left(t+ \frac {s \log (\lambda )} {K(s)}, \lambda x\right)
=e^{-t}e^{K(s)\left(t+ \frac {s \log (\lambda )} {K(s)}\right)-s(\log x+\log\lambda )}\\
=e^{-t}e^{K(s) t-s\log x}=u_s(t, x).
\end{eqnarray*}
These solutions are then self-similar. Notice although that they do not belong to $L^1(0, \infty))$ nor $L^1(x\, dx)$.
\end{remark}

\section{Proof of the main results}
\label{sec:main}
To study the long-time behaviour of the solutions $u$ of  Equation~(\ref{eq:frag})  we use their explicit representation~(\ref{def:invmellin}), that holds under the assumptions~(\ref{initialdata:int1}), (\ref{initialdata:int2}), (\ref{initialdata:int3}) on the initial data $u_0$ as shown in Theorem~\ref{thm:existence}. In order to simplify as much as possible the presentation we define the new function:
\begin{equation}
\label{link:uw}
w(t, x)=e^t u(t, x)
\end{equation}
that satisfies:
\begin{equation}
\label{eq:adim}
\f{\p}{\p t} w  =\int\limits_x^\infty \f{1}{y}k_0\left(\f{x}{y}\right)w(t,y)dy, \qquad w(0,x)=u_0(x),
\end{equation}
and, since ${ M}_{w}(0,s)=U_0(s),$ it is given by:
\begin{equation}
\label{def:wexplicit}
w(t, x)= \frac {1} {2\pi i}\int  _{ \nu -i\infty }^{ \nu +i\infty }x^{-s } U_0(s)e^{{K}(s) t}ds 
\end{equation}
for any  $\nu\in (1, 2)$.

\subsection{Behaviour for $x>1$.}
\label{xsmall}
The long-time behaviour of $u$ for $x>1$ is given in Theorem \ref{theorem1}. Its proof is based on the following lemma.

\begin{lemma}
\label{extq}
Suppose that $u_0$ satisfies~(\ref{initialdata:int1}), (\ref{initialdata:int2})-(\ref{initialdata:reg2}). Then the function $ U_0(s)$ defined by 
\begin{equation}
\label{def:explMellinw}
U_0(s)=\int \limits_0^\infty x^{s-1}u_0(x)dx,
\end{equation}
that is by hypothesis analytic on the strip $\Rea (s)\in (p_0, q_0)$ may be extended to a meromorphic function, still denoted $U_0(s)$, on $\Rea (s)\in (\rho , r )$ for $r >0$ given in~(\ref{initialdata:reg1}) and $\rho >0$ given in~(\ref{initialdata:reg1bis}). Moreover, 
\begin{equation}
\label{IntMellinw}
\int  _{ \nu -i\infty }^{ \nu +i\infty }|U_0(s)|ds<+\infty
\end{equation}
for all  $\nu \in (\rho , p_0)\cup(p_0, q_0)\cup(q_0, r)$.
\end{lemma}

\begin{proof}
By~(\ref{initialdata:int1}) we already have that $ U_0(s)$ is analytic on $\Rea (s) \in (p_0, q_0)$. For such an $s$ we may write:
\begin{eqnarray*}
U_0(s)=\int _0^1x^{s-1 }u_0(x)dx+\int _1^\infty x^{s-1} u_0(x)dx=I_1(s)+I_2(s).
\end{eqnarray*}

The integral $I_1(s)$ is analytic on the half plane $\Rea (s)> p_0$, and then, for the part $(p_0,r),$ it only remains to consider $I_2$. This term  may be written as follows
$$
I_2(s)=\int _1^\infty x^{s-1} a_0x^{-q_0}dx+
\int_1^\infty x^{s-1}(u_0(x)-a_0x^{-q_0})dx.$$
If $s\in \R^+$ and $s <q_0-1$, we have:
$$\int _1^\infty x^{s-1} (a_0x^{-q_0})dx=\frac {a_0} {q_0-s},
$$
and that function has a meromorphic extension to all the complex plane, with a single pole at $s =q_0$.  

On the other hand, by the hypothesis~(\ref{initialdata:reg1}),
$$
\int_1^\infty|x^{s-1} (u_0(x)-a_0x^{-q_0})|dx<C\int_1^\infty x^{-r+\Rea (s)-1}dx.
$$
The integral in the right-hand side converges for   $q_0\le \Rea (s ) < r$. We deduce that
$$
\int_1^\infty x^{s-1 }(u_0(x)-a_0x^{-q_0})dx
$$
is  analytic in the strip $\Rea (s) \in [q_0, r)$.

If we define now :
$$
H(s)=\int _0^1x^{s -1}u_0(x)-\frac {a_0} {s -q_0}+\int _1^\infty x^{s -1}(u_0(x)-a_0x^{-q_0}) dx.
$$
$H $ is meromorphic in the strip $\Rea (s )\in (p_0, q_0)$.  By definition of  $U_0(s)$ we have, for all  $\Rea (s )\in (p_0, q_0)$:
\begin{eqnarray*}
 U_0(s)=\int _0^1x^{s -1}u_0(x)-\frac {a_0} {s -q_0}+\int _1^\infty x^{s -1}(u_0(x)-a_0x^{-q_0}) dx.
\end{eqnarray*}
And in the strip $\Rea (s )\in (q_0, r)$ we have:
\begin{eqnarray*}
\frac {a_0} {s -q_0}=\int _0^1x^{s -1}a_0x^{-q_0}dx
\end{eqnarray*}
then, if $\Rea(s )\in (q_0, r)$:
$$
H _{  }
(s)=\int _0^\infty x^{s -1}(u_0(x)-a_0x^{-q_0}) dx:=G (s).
$$
We have then:
\begin{eqnarray*}
&&U_0(s)= H(s),\,\,\forall s ; \Rea(s )\in (a, q_0)\\
&&H(s)=G (s),\,\,\forall s ; \Rea(s )\in (q_0, r)
\end{eqnarray*}
and $G(s) $ is the analytic extension of  $ U_0(s) $ to $\Rea(s )\in (q_0, r)$. We still denote $ U_0(s) $ the meromorphic function that is obtained in that way  on $\Rea (s)\in(p_0, r)$.

We prove now~(\ref{IntMellinw}) for $\nu \in (p_0, r)$. By~(\ref{intMellinu}) we already know that it holds for $\nu \in (p_0, q_0)$. For all $s=u+iv$ in the strip $(q_0,r),$ we have:
\begin{eqnarray*}
 U_0(s)&=&\int _0^\infty x^{iv}x^{u-1} \left(u_0(x)-a_0x^{-q_0} \right)dx\\
&=&\left.\frac {1} {iv+1} x^{iv+1}x^{u-1}(u_0(x)-a_0x^{-q_0}) \right| _{ x=0 }^{x=\infty}-\\
&&-\frac {1} {iv+1} \int _0^\infty x^{iv+1}\left\{(u-1)x^{u-2}\left(u_0(x)-a_0x^{-q_0}\right)+\right.\\
&&\hskip 4cm \left.+x^{u-1}\left((u_0)'(x)+a_0q_0x^{-q_0-1} \right)\right\}dx.
\end{eqnarray*}
Since $u>q_0$, and using~(\ref{initialdata:int2}),  we have:
$$
|x^{iv+1}x^{u-1}(u_0(x)-a_0x^{-q_0})|\le x^{u}(|u_0(x)|+|a_0x^{-q_0}| ) \to 0 \,\,\,\hbox{as}\,\,x\to 0.
$$
On the other hand, using that $u<r$ and  (\ref{initialdata:reg1}):
$$
\left|x^{iv+1}x^{u -1}(u_0(x)-a_0x^{-q_0}) \right|\le Cx^{u-r}\to 0  \,\,\,\hbox{as}\,\,x\to +\infty.
$$
Therefore,
\begin{eqnarray*}
&&U_0(s)=-\frac {1} {iv+1} \int _0^\infty x^{iv+1}\left\{(u-1)x^{u-2}\left(u_0(x)-a_0x^{-q_0}\right)+\right.\\
&&\hskip 4cm \left.+x^{u-1}\left((u_0)'(x)+a_0q_0x^{-q_0-1} \right)\right\}dx\\
&&\hskip 0.9cm =-\frac {1} {iv+1} \int _0^\infty x^{iv}\left\{(u-1)x^{u-1}\left(u_0(x)-a_0x^{-q_0}\right)+\right.\\
&&\hskip 4cm \left.+x^{u}\left((u_0)'(x)+a_0q_0x^{-q_0-1} \right)\right\}dx\\
&&\hskip 0.9cm=J_1(s )+J_2(s ).
\end{eqnarray*}
Since
\begin{eqnarray*}
J_1(s)&=& -\frac {u-1} {iv+1}\int _0^\infty x^{iv}x^{u-1}\left (u_0(x)-a_0x^{-q_0}\right) dx\\
&=&-\frac {u-1} {iv+1}\int _0^\infty x^{s -1}\left (u_0(x)-a_0x^{-q_0}\right) dx\\
&=&-\f{u-1}{iv+1} U_0(s),
 \end{eqnarray*}
 we have:
\begin{eqnarray*}
-\frac {(iv+1)} {u-1}J_1(s)&=&J_1(s)+J_2(s)\\
\frac {(u+iv)} {u-1}J_1(s)&=&-J_2(s)\\
&=&\frac {1} {iv+1} \int _0^\infty x^{iv}\left\{x^{u}\left((u_0)'(x)+a_0q_0x^{-q_0-1} \right)\right\}dx\\
 J_1(s)&=&\frac {u-1} {(iv+1)(u+iv)} \int _0^\infty x^{iv}\left\{x^{u}\left((u_0)'(x)+a_0q_0x^{-q_0-1} \right)\right\}dx\\
|J_1(s)|&\le &\frac {|u-1|} {|(iv+1)(u+iv)|} \int _0^\infty \left|x^{u}\left((u_0)'(x)+a_0q_0x^{-q_0-1} \right)\right|dx.
\end{eqnarray*}

Therefore:
\begin{equation}
\label{est1}
U_0(s)=\frac {-1} {(u+iv)} \int _0^\infty x^{iv}\left\{x^{u}\left((u_0)'(x)+a_0q_0x^{-q_0-1} \right)\right\}dx.
\end{equation}
By (\ref{initialdata:int2}), (\ref{initialdata:reg1}) and~\eqref{initialdata:reg2}, and since $u>q_0$,  the integral in the right-hand side of (\ref{est1}) is absolutely convergent for all $s$ such that $\Rea (s)\in (q_0, r)$.

We may repeat the integration by parts with the integral in the right-hand side of (\ref{est1}):
\begin{equation}
J_3(s)=\int _0^\infty x^{iv}\left\{x^{u}\left((u_0)'(x)+a_0q_0x^{-q_0-1} \right)\right\}dx
\end{equation}
with $s=u+iv$ and $u\in (q_0, r)$.
As before,
\begin{eqnarray*}
&&J_3(s)=\left.\frac {1} {iv+1} x^{iv+1}x^{u}\left((u_0)'(x)+a_0q_0x^{-q_0-1} \right) \right| _{ x=0 }^{x=\infty}-\\
&& -\frac {1} {iv+1} \int _0^\infty x^{iv+1}\left\{ux^{u-1}\left((u_0)'(x)+a_0q_0x^{-q_0-1} \right)+\right.\\
&&\hskip 3cm \left.+x^{u}\left((u_0)''(x)-a_0q_0(q_0+1)x^{-q_0-2} \right)\right\}dx.
\end{eqnarray*}
By (\ref{initialdata:int3}) and the fact that $u>q_0$ we have:
\begin{eqnarray*}
|x^{iv+1}x^{u}\left((u_0)'(x)+a_0q_0x^{-q_0-1} \right)|\le x^{u+1}|(u_0)'(x)|+|a_0|q_0 x^{u-q_0}\to 0
\end{eqnarray*}
as $x\to 0$. Using now (\ref{initialdata:reg2})
\begin{eqnarray*}
|x^{iv+1}x^{u}\left((u_0)'(x)+a_0q_0x^{-q_0-1} \right)|\le Cx^{u+1-r-1}\to 0,\,\,\,\hbox{as}\,\,x\to \infty.
\end{eqnarray*}
Then,
\begin{eqnarray*}
J_3(s)&=&
 \frac {-1} {iv+1} \int _0^\infty x^{iv+1}\left\{ux^{u-1}\left((u_0)'(x)+a_0q_0x^{-q_0-1} \right)+\right.\\
 &&\hskip 4cm \left.+x^{u}\left((u_0)''(x)-a_0q_0(q_0+1)x^{-q_0-2} \right)\right\}dx\\
& =&\frac {-1} {iv+1} J_3(s)- \frac {1} {iv+1} \int _0^\infty x^{s+1}\left((u_0)''(x)-a_0q_0(q_0+1)x^{-q_0-2} \right)dx
\end{eqnarray*}

\begin{eqnarray*}
 J_3(s) &=&-\frac {1} {iv+2} \int _0^\infty x^{s+1}\left((u_0)''(x)-a_0q_0(q_0+1)x^{-q_0-2} \right)dx\\
 &=&-\frac {1} {iv+2}\left( \int _0^1 (\cdots)dx+ \int _1^\infty (\cdots)dx\right).
\end{eqnarray*}
The first integral in the right-hand side is absolutely convergent by (\ref{initialdata:int3}) and because $\Rea (s)-q_0-1>-1$. The second integral converges absolutely by 
the hypothesis~(\ref{initialdata:reg3}).

We have then, for $\Rea (s)\in (q_0, r)$:
$$
U_0(s)=\frac {1} {(u+iv)(iv+2)} \int _0^\infty x^{s+1}\left((u_0)''(x)-a_0q_0(q_0+1)x^{-q_0-2} \right)dx
$$
and  (\ref{IntMellinw}) follows for $\Rea (s)\in (q_0, r)$.

A similar argument, using~(\ref{initialdata:reg1bis}), (\ref{initialdata:reg2bis}) and (\ref{initialdata:reg3bis}), shows (\ref{IntMellinw}) also for $\Rea (s)\in (\rho ,  p_0)$.
\end{proof}  
We may prove now  Theorem \ref{theorem1}.
\begin{proof}[Proof of Theorem \ref{theorem1}.]
Let us consider the function $w(t, x)$ defined in (\ref{link:uw}) and given by (\ref{def:wexplicit}) for all $\nu \in (p_0, q_0)$. By Lemma \ref{extq} we may deform the integration contour in (\ref{def:wexplicit}), cross the straight line $\Rea (s)=q_0$ to obtain:
\begin{equation}
\label{wpoleq}
w(t, x)= Res\left(U_0(s); s=q_0 \right) x^{-q_0}e^{{2K }(q_0)t}+ \frac {1} {2\pi i}\int  _{ \nu' -i\infty }^{ \nu' +i\infty }x^{-s } U_0(s)e^{{K}(s) t}ds 
\end{equation}
with $\nu'\in (q_0,r)$.
  We notice now,  for $s=\nu'+iv:$
 \begin{eqnarray*}
&&{K}(s)= \int\limits_0^1 x^{\nu'}cos(v \log(x))k_0(x) dx + i \int\limits_0^1 x^{\nu'}sin(v \log(x))k_0(x) dx \\
&&e^{{K}(s)t}=e^{t \int\limits_0^1 x^{\nu'}cos(v \log(x))k_0(x) dx} \, e^{ it 
\int\limits_0^1 x^{\nu'}sin(v \log(x))k_0(x) dx }
\end{eqnarray*}

\begin{eqnarray}
&&\left|
\int  _{ \nu' -i\infty }^{ \nu' +i\infty }x^{-s } U_0(s)e^{{K}(s)t}d\alpha\right|
\le  e^{t \int\limits_0^1 x^{\nu'}k_0(x) dx}\times\\
&&\hskip 2cm \times \left|\int  _{ \nu' -i\infty }^{ \nu' +i\infty }x^{-s } U_0(s) e^{ it 
\int\limits_0^1 x^{\nu'}sin(v \log(x))k_0(x) dx}d\alpha \right|  \nonumber\\
&&\le   x^{-\nu' }e^{t {K}(\nu')} \left|\int  _{ -\infty }^{+\infty }x^{i v}U_0(\nu'+iv) e^{it 
\int\limits_0^1 x^{\nu'}sin(v \log(x))k_0(x) dx}dv\right| \nonumber\\
&&\le 
x^{-\nu' }e^{t {K}(\nu')}\int  _{ -\infty }^{+\infty }\left|U_0(\nu'+iv)\right|dv. \label{estres}
\end{eqnarray}
 Combining (\ref{wpoleq}) and  (\ref{estres}) we obtain, for $x>1$ and $t\to \infty$:
 \begin{eqnarray}
 w(t, x)=Res\left(U_0(s); s=q_0 \right) x^{-q_0}e^{{K }(q_0)t}
 \left(1+\mathcal O\left(x^{-\nu'+q_0} e^{(K(\nu')-K(q_0))t}\right) \right).
 \end{eqnarray}
 \end{proof}
\subsection{$x$ small and $t$ large}
In order to understand the long-time behaviour of the function $u$, or $w$, let us recall the definition~\eqref{def:phi2} of the key function $\phi:$
\begin{equation}
\label{def:phi}
\phi(s,t,x)=-s \log(x) + t {K}(s).
\end{equation}
Let us consider for a while the function $\phi $ only for real values of $s$. 
When $t>0$ and $x>1$, this function is decreasing with respect to the real part of $s$. That is the key property that, in Subsection~\ref{xsmall},  makes the rest term to be negligible with respect to the residue at $s=q_0$, once we have that $U_0(s)$ is integrable for $s \in (q_0, r)$. If we maintain $t$ fixed and take $x<1$, the function $\phi$ is now increasing with respect to $s$. We could then easily obtain a result  similar to Theorem \ref{theorem1} for $t$ fixed and $x\to 0$.  But that is not a region particularly interesting.  These monotonicity properties of $\phi $ do not hold exactly when $t$ is large and $x$ small although that is exactly the interesting domain to consider for the long-time behaviour of $u$. Lemma~\ref{lem:phi} shows the behaviour of $\phi(\cdot, t, x):$ it is decreasing for $s \in (p_1,s_+(t,x))$ and increasing for $s \in (s_+(t,x),+\infty)$ where $s_+(t,x)$ is defined by~\eqref{alphaplus2}.

Since $s_+$ depends on  $t$ and $x,$ is an increasing function of  $t>0$ and an increasing function of $x\in (0, 1)$ we have several domains of different long-time asymptotic behaviour, that correspond to the different points (i) and (ii)  of Theorem \ref{theorem2} and Theorem \ref{theorem3}.

\begin{proof}[Proof of Theorem \ref{theorem2}]
We only consider in detail  the case (i) since the case (ii) is completely similar.  Suppose then  that we are in the region where $t\to \infty$ and $x\to 0$ in such a way that 
$$s_+(t, x)=(K')^{-1}\left(\f{\log(x)}{t}\right)>q_0.$$ By (\ref{def:wexplicit}):
\begin{equation}
\label{def:wphi}
w(t, x) =
\frac {1} {2\pi i}\int  _{ \nu -i\infty }^{ \nu +i\infty }U_0(s)e^{\phi (s, t, x)}ds
\end{equation}
for any  $\nu\in (p_0, q_0)$. 

Using Lemma \ref{extq} we deform the integration contour  in (\ref{def:wphi}) towards the right and cross the pole $q_0$, i.e. that we have:
\begin{eqnarray*}
w(t,x)
=  e^{\phi(q_0,t,x)}\bigl(Res\left(U_0(s); s=q_0 \right)+ \mathcal R\bigr)\\
\mathcal R=  \frac {1} {2\pi i}\int  _{ \nu' -i\infty }^{ \nu' +i\infty }U_0(s)e^{-\phi(q_0,t,x)+\phi (s, t, x)}ds
\end{eqnarray*}
for any $\nu'\in (q_0, s_+(t, x))$. Then,
\begin{eqnarray*}
\left|\mathcal R\right|\le
\int  _{ \nu' -i\infty }^{ \nu' +i\infty }\left|U_0(s)\right| e^{\Rea(-\phi(q_0,t,x)+\phi(s,t,x))}ds,
\end{eqnarray*}

where
\begin{eqnarray*}
\Rea(-\phi(q_0,t,x)+\phi(s,t,x))=\Rea \left( - (s -q_0) \log x+t(K(s)-K ( q_0 ))\right)
\\
= -\delta \log x+t\Rea \left(\int _0^1 z^{s -1}k_0(z)dz-\int _0^1 z^{q_0 -1}k_0(z)dz\right)\\
=-\delta \log x+t\left(\int _0^1 z^{\nu' -1}\cos(v\log z)k_0(z)dz-\int _0^1 z^{q_0 -1}k_0(z)dz\right)\\
\le -\delta \log x+t\left(\int _0^1 z^{\nu' -1}k_0(z)dz-\int _0^1 z^{q_0 -1}k_0(z)dz\right)\\
=-\phi(q_0,t,x)+\phi(\nu'  ,t,x).
\end{eqnarray*}
By Taylor's expansion
\begin{eqnarray*}
\phi (s_+(t, x),t,x)-\phi (q_0,t,x)=-\frac {1} {2}(s_+(t, x)-q_0)^2 \f{\p^2 \phi}{\p s^2} (\xi ,t,x)\\
=-\frac {1} {2}(s_+(t, x)-q_0)^2tK ''(\xi, t, x),
\end{eqnarray*}
for some $\xi \in (q_0, s_+(t, x))$. Since $K'''<0$ we deduce that $K''(\xi )\ge K''(s_+(t, x)\ge K''(q_0)$ and then 
$$
\phi (s_+(t, x))-\phi (q_0)\le -\frac {1} {2}(s_+(t, x)-q_0)^2tK ''(q_0).
$$
It follows that
$$
|\mathcal R|\le e^{-\frac {1} {2}(s_+(t, x)-q_0)^2tK ''(q_0)}\int  _{ \nu' -i\infty }^{ \nu' +i\infty }\left|U_0(s)\right| ds,
$$
and the point (ii) of Theorem~\ref{theorem2} follows. The proof of the point (i) is similar. 
\end{proof}

\subsection{Proof of  Theorem \ref{theorem3}.}

In order to estimate the function $w(t, x)$ defined in (\ref{def:wexplicit}) when $s_+(t, x)\in (p_0, q_0)$ we  take $\nu=s_+(t, x)$ in that formula. Then, we will use several  estimates on $U_0(s)$ and  $\Rea \big(\phi (s, t, x)\big)$ along the integration curve 
$\Rea (s)=s_+(t, x)\in (p_0+\delta,q_0-\delta)$.

The estimate on $U_0(s)$ follows from (\ref{intMellinu}):
\begin{eqnarray}
|U_0(s_++iv)|&\le&\frac {1} {\sqrt{(s_+^2+v^2)((1+s_+)^2+v^2)}}\int _0^\infty |(u_0)''(y)|y^{1+s_+}dy \nonumber \\
&\le& \frac {1} {\sqrt{(p_0^2+v^2)((1+p_0)^2+v^2)}}\times\nonumber \\
&\times &\hskip -0.3cm\left(\int _0^1 |(u_0)''(y)|y^{1+p_0+\delta}dy+\int_1^\infty |(u_0)''(y)|y^{1+q_0-\delta}dy\right). \label{intU0bis}
\end{eqnarray}
We then have
\begin{eqnarray}
\int  \limits_{\Rea (s)=s_+, |\Ima (s)|\ge \varepsilon }\!\!\!\!\!\!\left|U_0(s) \right|dv\le \int  _{\Rea (s)=s_+}\frac {dv} {\sqrt{(p_0^2+v^2)((1+p_0)^2+v^2)}}
\times \nonumber \\
\times \left(\int _0^1 |(u_0)''(y)|y^{1+p_0+\delta}dy+\int_1^\infty |(u_0)''(y)|y^{1+q_0-\delta}dy\right)=C_0(p_0+\delta, q_0-\delta). \label{intU0}
\end{eqnarray}
In order to estimate $\Re e(\phi )$ we wish to use the  method of stationary phase. To this end we must consider different cases, depending on the measure $k_0$. 

\begin{lemma} For all $\varepsilon \in (0, 1)$, for all $t>0$, $x>0$ and all $s\in \C$ such that $s=s_+(t,x) + iv,$ $v\in \R$ and $|s-s_+(t, x)|\le \varepsilon $:
\label{statphase2}
\begin{equation}
\label{statphase2:1}
\Rea \big( K(s) \big)=K (s_+(t, x)) - \f{v^2}{2} 
K''(s_+(t, x))+\mathcal O\left(\varepsilon ^3 \right).
\end{equation}
If  $k_0$ is a discrete measure whose support $\Sigma$ satisfies Condition~H, then (\ref{statphase2:1}) holds for all  $\varepsilon >0$, $t>0$, $x>0$ and for all $s\in s_+ + i\R$ such that $|s-(s_+(t, x)+iv)|\le \varepsilon $,  for some $v\in Q$.
\end{lemma}
\begin{proof}
By Taylor's formula
\begin{eqnarray*}
&&K(s)=K (s_+(t, x))+(s-s_+(t, x))K'(s_+(t, x))+\frac {(s-s_+(t, x))^2} {2}K''(s_+(t, x))\\
&&\hskip 3cm +\frac {(s-s_+(t, x))^3} {6} \mathcal  K'''(\zeta (t, x))
\end{eqnarray*}
for some $\zeta (t, x)\in\mathbb C$ between $s$ and $s_+(t, x)$.  Since $|s-s_+(t, x)|\le \varepsilon $, $|\zeta (t, x)-s_+(t, x)|\le \varepsilon $. Since $s_+\in [p_0, q_0]$ it follows that   $|\zeta (t, x)|\le q_0+\varepsilon $.
We have then
\begin{eqnarray*}
|K'''(\zeta)|&=& \left| \int _0^1 (\log z )^3 k_0(z) e^{(\zeta -1)\log z}dz\right|\\
&\le&  \int _0^1 |\log z |^3 |k_0(z)| e^{\Rea(\zeta -1)\log z}dz\\
&\le &\int _0^1 |\log z |^3 |k_0(z)| e^{-(p_0+1)\log z}dz.
\end{eqnarray*}
Using that $K'(s_+) \in \R$, (\ref{statphase2:1}) follows.

Suppose now that $k_0$ satisfies Condition~H. We  have:
\begin{eqnarray*}
&&k_0(x)=\sum _{ k\in \N }a _{ k }\delta (x-\sigma _k), \,\,\sum _{ k\in \N }a _{ k }\sigma _k=1\\
&&K(s)=\sum _{ \ell\in \N }a _{ k }\sigma _\ell^{s-1},\,\,K'(s)=\sum _{ \ell\in \N }a _{ k }(\log \sigma )\sigma _\ell^{s-1}\\
&&\phi (s, t, x)=-s\log x+t\sum _{ \ell\in \Z }a _{ k\ell}\sigma _\ell^{s-1}.
\end{eqnarray*}
Therefore, for all $v\in Q$
\begin{eqnarray}
&&K(s_++iv)= \sum _{ \ell\in \N }a _{ \ell }\sigma _\ell^{s_+-1}
e^{iv \log \sigma _\ell} \nonumber \\
&&= \sum _{ \ell\in \N }a _{ \ell }\sigma _\ell^{s_+-1}e^{2i\pi k(\ell)}=
 \sum _{ \ell\in \N }a _{ \ell }\sigma _\ell^{s_+-1}=K(s_+), \label{Kegal}
\end{eqnarray} 
and for the same reason
\begin{eqnarray*}
K' (s_++iv)=K'(s_+),\,\,\forall v\in Q,\\
K'' (s_++iv)=K'' (s_+),\,\,\forall v\in Q,
\end{eqnarray*}
from where (\ref{statphase2:1}) follows for  all $s\in \C$ such that $|s-(s_+(t, x)+iv)|\le \varepsilon $ for some $v\in Q$ by Taylor's formula around the point $s_+(t, x)+iv$.
\end{proof}

We prove now estimates on $\Rea (\Phi )$ for $s$ far from the points $s_k=s_+ +k v_*$ with $v_*$ defined by Equation~\ref{def:vstar}.
\begin{lemma}
\label{statphase1}
Suppose that the measure $k_0$ has an absolutely continuous part or  is a singular discrete measure whose support $\Sigma$ does not satisfy Condition~H. Then, for all $\varepsilon >0$ there exists $\gamma (\varepsilon )>0$ such that for all $t>0$ and $x>0$:
\begin{equation*}
\sup _{ |v|\ge \varepsilon  }\Rea \biggl(\phi (s_++iv, t, x)\biggr)\le \Rea \biggl(\phi (s_+, t, x)\biggr)-\gamma (\varepsilon ) t.
\end{equation*}
Moreover, $\lim _{ \varepsilon \to 0 }\varepsilon ^{-2}\gamma (\varepsilon )=C$ for some constant $C\in (K^{''} (q_0) /2 ,  K^{''} (p_0)/2)$.
\end{lemma}
\begin{proof} For all $v\in \R$:
\begin{eqnarray*}
\Rea \biggl(\phi (s_++iv, t, x)-\phi (s_+, t, x)\biggr)&=&\Rea\biggl(-iv \log(x) + t (K(s_+ + iv)-K(s_+))\biggr)\\
&=& t \Rea(K(s_+ + iv)-K(s_+))\\
&=& t\Rea\biggl( \int\limits_0^1 x^{s_+-1} k_0(x) (x^{iv}-1)dx\biggr)\\
&=&\int\limits_0^1 x^{s_+-1} k_0(x) \biggl(cos\big(v\log(x)\big)-1\biggr)dx\\
&=& t G(v)\leq 0.
\end{eqnarray*}
The function $G(v)$ is a continuous nonpositive function of $v$.
If $k_0$ has an absolutely continuous part $g$ then:
$$\forall\;v\in \R \backslash \{0\},\qquad\int\limits_0^1 x^{s_+-1} cos(v\log x) g(x)dx < \int\limits_0^1 x^{s_+-1} g(x)dx,$$
and then, for all $v\in \R \backslash \{0\}$:
\begin{eqnarray*}
\int\limits_0^1 x^{s_+-1}  cos(v\log x) k_0(x)dx&=& \int\limits_0^1 x^{s_+-1}  cos(v\log x) (g(x)+d\mu (x))dx\\
 &<& \int\limits_0^1 x^{s_+-1} (g(x)dx +d\mu(x))=\int\limits_0^1 x^{s_+-1} k_0(x)dx.
\end{eqnarray*}
By continuity of the function $G(v)$ it follows that for all $R>0$ and $\varepsilon >0$ there exists $\delta >0$ such that
$$
\sup _{ \varepsilon \le |v|\le R } G(v)\le -\delta .
$$
Lemma~(\ref{statphase1}) will follow if we prove  that the limit of $G$ at infinity is also strictly negative. To this end let us write $G$ under the form
\begin{eqnarray*}
G(v)&=&\int\limits_0^1 x^{s_+-1} k_0(x) \biggl(cos\big(v\log(x)\big)-1\biggr)dx\\
&=&\int\limits_0^\infty  e^{-ys_+} k_0(e^{-y}) \biggl(cos\big(vy\big)-1\biggr)dy \\
&=&\int\limits_0^\infty  e^{-ys_+} \biggl(cos\big(vy\big)-1\biggr)\biggl(g(e^{-y})dy +d\mu(e^{-y})\biggr).
\end{eqnarray*}
By Riemann Lebesgue theorem, 
$$\lim\limits_{v\to\pm\infty}\int\limits_0^\infty  e^{-ys_+} cos\big(vy\big) g(e^{-y})dy=0,$$
and then
\begin{eqnarray*}
\lim\limits_{v\to\pm\infty}G(v)&=&\lim\limits_{v\to\pm\infty}\int\limits_0^\infty  e^{-ys_+} \biggl(cos\big(vy\big)-1\biggr) \biggl(g(e^{-y})dy +d\mu(e^{-y})\biggr)
\\
&=&-\int\limits_0^\infty e^{-ys_+} g(e^{-y})dy +\lim\limits_{v\to\pm\infty}\int\limits_0^\infty  e^{-ys_+} \biggl(cos\big(vy\big)-1\biggr) d\mu(e^{-y})\\
&\le& -\int\limits_0^\infty e^{-ys_+} g(e^{-y})dy =-\int\limits_0^1 x^{s_+-1} g(x)dx <0. 
\end{eqnarray*}

Suppose now that the measure $k_0$ is a singular discrete measure whose support $\Sigma$ is given by the countable set of points $\sigma $:
\begin{eqnarray*}
&&k_0(x)=\sum _{ \ell\in \N }a _{ k }\delta (x-\sigma _\ell), \,\,\sum _{ \ell\in \N }a _{ \ell }\sigma _\ell=1\\
&&K(s)=\sum _{ \ell\in \N }a _{ k }\sigma _\ell^{s-1},\,\,K'(s)=\sum _{ \ell\in \N }a _{ k }(\log \sigma )\sigma _\ell^{s-1}\\
&&\Phi (s, t, x)=-s\log x+t\sum _{ \ell\in \N }a _{ k\ell}\sigma _\ell^{s-1}.
\end{eqnarray*}
By definition of $p_1$:
$$
\sum _{ \ell\in \N }a _{ \ell }\sigma^s _\ell<+\infty,\,\,\,\forall s; \,\,\Rea s> p_1.
$$
Suppose also that the points $\sigma _\ell$  do not satisfy Condition~H.  Then,  for any $v\not =0$, 
$$
\int\limits_0^1 x^{s_+-1} k_0(x) \cos\big(v\log(x)\big)<\int\limits_0^1 x^{s_+-1} k_0(x) \cos\big(v\log(x)\big).
$$
Suppose on the contrary that this is not true. Then, by the definition of $k_0$:
$$
\sum _{ \ell  \in \N } a _{ \ell  }\sigma _{ \ell } ^{s_+-1} \cos\big(v\log \sigma _{ \ell } \big)= 
\sum _{ \ell  \in \N } a _{ \ell  }\sigma _{ \ell } ^{s_+-1}.
$$
Since $a _{ \ell  }\ge 0$, $\sigma  _{ \ell }\ge 0$ and $\sigma _\ell$  we have
$$
\cos\big(v\log \sigma _{ \ell } \big)=1,\,\,\,\forall \ell \in \N,
$$
or, in other terms,
$$
\forall \ell\in \N, \exists k(\ell)\in \N; \,\,\,v\log \sigma _{ \ell } =2k(\ell) \pi .
$$
But, by Proposition \ref{propositioncondH1} this contradicts the assumption that $\Sigma$ does not satisfy Condition~H.  Suppose now once again that there exists a sequence $v_n\to \infty$ such that
$$
\int\limits_0^1 x^{s_+-1} k_0(x) \cos\big(v_n\log(x)\big)=\sum _{ \ell  \in \N } a _{ \ell  }\sigma _{ \ell } ^{s_+-1} \cos\big(v_n\log \sigma _{ \ell } \big)\to 
\sum _{ \ell  \in \N } a _{ \ell  }\sigma _{ \ell } ^{s_+-1}.
$$
If  $\min  _{ \ell }\sigma _\ell=\sigma  _{ \ell_0 }>0$, for all $n$ there exists $w_n\in [0, 2\pi /\log \sigma  _{ \ell_0 }]$ such that 
$$\cos\big(w_n\log \sigma _{ \ell } \big)=\cos\big(v_n\log \sigma _{ \ell } \big)\,\,\,\,\,\forall \ell. $$
Then, there exists a subsequence $(w_n) _{ n\in \N }$ and $w_*\in  [0, 2\pi /\log \sigma  _{ \ell_0 }]$ such that $w_n\to w_*$ as $n\to \infty$. By continuity this would imply
$$
\sum _{ \ell  \in \N } a _{ \ell  }\sigma _{ \ell } ^{s_+-1} \cos\big(w_*\log \sigma _{ \ell } \big)=\sum _{ \ell  \in \N } a _{ \ell  }\sigma _{ \ell } ^{s_+-1}.
$$
Since $a _{ \ell }\ge 0$ and $\sigma _\ell \ge 0$ this implies that:
$$
\cos\big(w_*\log \sigma _{ \ell } \big)=1,\,\,\,\forall \ell \in \N,
$$
which is impossible since the set of the points $\sigma  _{ \ell }$ does not satisfies Condition~H. Therefore 
$$
\sup _{ |v|\ge \varepsilon  }\int _0^1x^{s_+ -1}\cos(v\log x) d\mu _c < \int _0^1x^{s_+ -1} d\mu _c
$$
and this concludes the proof of Lemma  \ref{statphase1}. As concerns the asymptotic behaviour when $\varepsilon \to 0,$ it is given by Taylor's formula of Lemma~\ref{statphase2}, since $$0< K^{''} (q_0) < K^{''} (s_+)<K^{''} (p_0),$$ so that $\gamma\sim C \varepsilon^2$ with $K^{''} (q_0) /2 < C < K^{''} (p_0)/2.$ 
\end{proof}
\begin{remark}
\label{rem:contsing}
We do not know whether Lemma \ref{statphase1} holds when the absolutely continuous part of the measure $k_0$ is zero  but its  singular continuous part is not zero. For our purpose it would be enough to know if, given a singular continuous probability measure $d\mu$ with support contained in $[0, 1]$:
$$
\overline{\lim} _{ v\to \infty }\int _0^1e^{ivx}d\mu (x)<1.
$$
\end{remark}

\begin{remark}
\label{rem:condH} As it is seen in the proof of Lemma \ref{statphase1},  Condition~H arises very naturally when looking for the existence of at least one point $v\in \R$ such that for all $\ell \in \N$ there exists $k(\ell)\in \Z$ such that $v\log \sigma  _{ \ell }=2\pi k(\ell)$. For more details see Proposition~\ref{propositioncondH1} in the appendix.
\end{remark}

We now consider the case where the measure $k_0$ is discrete and its support satisfies Condition~H. Let us first observe the following.

\begin{lemma}
\label{statphase3}
Suppose that  $k_0$ is a singular discrete measure whose support $\Sigma$ satisfies Condition~H. Then, 
\begin{eqnarray}
\label{statphase3:1}
&& \forall \varepsilon \in \left(0,|\log\theta|^{-1} \right) , \exists \gamma(\varepsilon , \theta) >0, \qquad \gamma \sim_{\varepsilon \to 0} C \ep^2,\quad C>0; \nonumber\\
&& \Rea \biggl(\phi (s_++iv, t, x)\biggr)\le \Rea \biggl(\phi (s_+, t, x)\biggr)-\gamma (\varepsilon, \theta ) t,\\
&&\forall v\in \R\setminus\{0\}, \hbox{such that}, d(v, \,Q)\ge \varepsilon.
\end{eqnarray}
\end{lemma}
\begin{proof}
Since  $k_0$ is a discrete measure whose support satisfies Condition~H,
\begin{eqnarray*}
 \Rea \biggl(\phi (s_++iv, t, x)\biggr)&=&-s_+\log x+t\Rea\biggl(\sum _{ \ell\in \N }a _{ \ell}\sigma _\ell^{s-1}\biggr)\\
 &=&-s_+\log x+t\sum _{ \ell\in \N }a _{ \ell}e^{(s_+-1)\log \sigma  _{ \ell }}\cos(v\log \sigma  _{ \ell }).
\end{eqnarray*}
We now apply Proposition~(\ref{propositioncondH3}) with $b _{ \ell }=a _{ \ell}e^{(s_+-1)\log \sigma  _{ \ell }}$.
Then for all $\varepsilon >0$ there exists $\gamma =\gamma (\varepsilon , \theta) >0$ such that
$$
\sum _{ \ell\in \N }a _{ \ell}e^{(s_+-1)\log \sigma  _{ \ell }}\cos(v\log \sigma  _{ \ell })\le \sum _{ \ell\in \N }a _{ \ell}e^{(s_+-1)\log \sigma  _{ \ell }}-\gamma  
$$
and therefore
$$
\Rea \biggl(\phi (s_++iv, t, x)\biggr)\le \Rea \biggl(\phi (s_+, t, x)\biggr)-\gamma  t.
$$
The asymptotic behaviour of $\gamma(\ep,\theta)$ when $\ep\to 0$ comes from Taylor's formula given in Lemma~\ref{statphase2}.
\end{proof}
We prove now  Theorem \ref{theorem3}.
\begin{proof}[Proof of Theorem \ref{theorem3}.]  We start proving the point (a).

Let us split the integral defining $w(t, x)$ as follows:
\begin{eqnarray}
\label{wphi:split}
w(t, x)& = &
\frac {1} {2\pi i}\int  _{ s_+(t, x) -i\infty }^{ s_+(t, x) +i\infty }U_0(s)e^{\phi (s, t, x)}ds \nonumber \\
& = &\frac {1} {2\pi i}\int  _{ s_+(t, x) -i\varepsilon  }^{ s_+(t, x) +i\varepsilon  }U_0(s)e^{\phi (s, t, x)}ds+\nonumber\\
&&+\frac {1} {2\pi i}\int  _{\Rea (s)=s_+, |\Ima (s)|\ge \varepsilon }U_0(s)e^{\phi (s, t, x)}ds.
\end{eqnarray}

We estimate the two integrals in the right hand side of (\ref{wphi:split}) using Lemma~\ref{statphase2} and Lemma~\ref{statphase1}. First of all we consider the integral near $s_+$, and write:
\begin{eqnarray}
&&\frac {1} {2\pi i}\int  _{ s_+(t, x) -i\varepsilon  }^{ s_+(t, x) +i\varepsilon  }U_0(s)e^{\phi (s, t, x)}ds=
\frac {U_0(s_+)} {2\pi i}\int  _{ s_+(t, x) -i\varepsilon  }^{ s_+(t, x) +i\varepsilon  }e^{\phi (s, t, x)}ds+\\
&&\hskip 3cm +\frac {1} {2\pi i}\int  _{ s_+(t, x) -i\varepsilon  }^{ s_+(t, x) +i\varepsilon  }\left(U_0(s)-U_0(s_+)\right)e^{\phi (s, t, x)}ds\\
 && \hskip 4.8cm=I_1+I_2.\label{wphi:splitsplit}
\end{eqnarray}
By Lemma~\ref{statphase2}, for all $\varepsilon \in (0, 1)$, $x>0$, $t>0$ and $s=s_++iv$
\\
\noindent
such that $|s-s_+(t, x)|\le \varepsilon $,
\begin{equation}
\label{taylorphi }
\Re e\big(\phi (s, t, x)\big)=\Re e\big(\phi (s_+(t, x), t, x)\big)+t\left(\frac {(s-s_+(t, x))^2} {2}K''(s_+(t, x))+\mathcal O\left(\varepsilon ^3 \right)\right).
\end{equation}
and
$$
\int  _{ s_+(t, x) -i\varepsilon  }^{ s_+(t, x) +i\varepsilon  }e^{\phi (s, t, x)}ds=
e^{\phi (s_+(t, x), t, x)}\int  _{ s_+(t, x) -i\varepsilon  }^{ s_+(t, x) +i\varepsilon  }e^{t\left(\frac {(s-s_+(t, x))^2} {2}K''(s_+(t, x))+\mathcal O(\varepsilon ^3)\right)}ds.
$$
Since $s=s_+(t, x)+iv$,  then $(s-s_+(t, x))^2=-v^2$ and
\begin{eqnarray*}
\int  _{ s_+(t, x) -i\varepsilon  }^{ s_+(t, x) +i\varepsilon  }e^{\phi (s, t, x)}ds
&=&ie^{\phi (s_+(t, x), t, x)}\int  _{ -\varepsilon  }^{ \varepsilon  }e^{-t\left(\frac {v^2} {2}K''(s_+(t, x))+\mathcal O(\varepsilon ^3)\right)}dv\\
&=&i \frac {e^{\phi (s_+(t, x), t, x)}} {\sqrt{t K''(s_+)}}
\int  _{ -\varepsilon \sqrt{t K ''(s_+)+\mathcal O(\varepsilon )} }^{ \varepsilon \sqrt{t K ''(s_+) +\mathcal O(\varepsilon )}}e^{-\frac {y^2} {2}}dy.
\end{eqnarray*}
Since $s_+\in (p_0, q_0)$, $K''(s_+) \in (K''(q_0),K''(p_0))$ and $\varepsilon \sqrt{t K ''(s_+)+\mathcal O(\varepsilon )}\to \infty$ as $t\to \infty$ and $\varepsilon \to 0,$ if we choose them such that $\varepsilon \sqrt{t} \to \infty$
\begin{eqnarray*}
&&\int  _{ -\varepsilon \sqrt{t K ''(s_+)+\mathcal O(\varepsilon )} }^{ \varepsilon \sqrt{t K ''(s_+)+\mathcal O(\varepsilon )} }e^{-\frac {y^2} {2}}dy=
\int  _{ -\infty }^{ \infty}e^{-\frac {y^2} {2}}dy-2\int  _{\varepsilon \sqrt{t K ''(s_+)+\mathcal O(\varepsilon )}   }^\infty e^{-\frac {y^2} {2}}dy\\
&&\hskip 2.5cm =\sqrt {2\pi}-2e^{-\frac {t\left( \varepsilon^2K ''(s_+)+\mathcal O\left(\varepsilon ^3\right)\right)} {2}}\times \mathcal O \left(\frac {1} { \varepsilon \sqrt{t K ''(s_+)+\mathcal O(\varepsilon )}} 
\right),
\end{eqnarray*}
as $t\to \infty$, $\varepsilon \to 0,$ $\varepsilon \sqrt{t} \to \infty$ and $s_+(t, x)\in (p_0, q_0)$.
\begin{eqnarray*}
\frac {1} {2\pi i}\int  _{ s_+(t, x) -i\varepsilon  }^{ s_+(t, x) +i\varepsilon  }e^{\phi (s, t, x)}ds=
\frac{1}{\sqrt{2 \pi }} \frac {e^{\phi (s_+(t, x), t, x)}} {\sqrt{t K''(s_+)}}
\left(1+\mathcal O \left( 
 \frac {e^{-\frac {t\left( \varepsilon^2K ''(s_+)+\mathcal O\left(\varepsilon ^3\right)\right)} {2}}} { \varepsilon \sqrt{t K ''(s_+)+\mathcal O(\varepsilon )}}\right) \right).
\end{eqnarray*}
Since $K'''(s)\le 0$ and $s_+(t, x)\in (p_0, q_0)$, $K''(s_+)\ge K''(q_0)$ and then,
\begin{eqnarray}
I_1= \frac{u_0(s_+(t, x))}{\sqrt{2 \pi }} \frac {e^{\phi (s_+(t, x), t, x)}} {\sqrt{t K''(s_+)}}\left(1+\mathcal O \left( 
 \frac {e^{-\frac { \varepsilon^2 t K ''(q_0)} {2}}} { \varepsilon \sqrt{t K ''(q_0)}}\right)\right) \label{est:I1}.
\end{eqnarray}

We consider now $I_2$. 
\begin{eqnarray*}
&&\left|\int  _{ s_+(t, x) -i\varepsilon  }^{ s_+(t, x) +i\varepsilon  }\left(U_0(s)-U_0(s_+)\right)e^{\phi (s, t, x)}ds\right|\le\\
&&\le \sup _{ |v|\le \varepsilon }\left|U_0(s_++iv)-U_0(s_+)\right|
\int  _{ s_+(t, x) -i\varepsilon  }^{ s_+(t, x) +i\varepsilon  }\left|e^{\phi (s, t, x)}\right|ds.
\end{eqnarray*}
As before in Lemma~\ref{statphase2}, if $s=s_+(t, x)+iv$ with $v\in (-\varepsilon , \varepsilon )$:
$$
 \left|e^{\phi (s, t, x)}\right|=e^{\phi (s_+(t, x), t, x)}e^{-t\left( \frac {v^2} {2}K''(s_+(t, x))+\mathcal O(\varepsilon ^3)\right)},
$$
and we may then estimate $I_2$ as
\begin{eqnarray}
|I_2|&\le& \omega (\varepsilon )
 \frac {e^{\phi (s_+(t, x), t, x)}} {\sqrt{2\pi t K''(s_+)}}\left(1+\mathcal O \left( 
 \frac {e^{-\frac {t\left( \varepsilon^2K ''(s_+)+\mathcal O\left(\varepsilon ^3\right)\right)} {2}}} { \varepsilon \sqrt{t K ''(s_+)+\mathcal O(\varepsilon )}}\right) \right) \nonumber \\
 &=& \omega (\varepsilon ) \frac {e^{\phi (s_+(t, x), t, x)}} {\sqrt{2\pi t K''(s_+)}}
\left(1+\mathcal O \left( 
 \frac {e^{-\frac { \varepsilon^2 t K ''(q_0)} {2}}} { \varepsilon \sqrt{t K ''(q_0)}}\right)\right).  \label{est:I2}
\end{eqnarray}
where $\omega (\varepsilon )$ is defined by 
\begin{equation}\label{eq:defomega}
\omega(\varepsilon):=\sup _{ |v|\le \varepsilon }\left|U_0(s_++iv)-U_0(s_+)\right|.
\end{equation}

Moreover, for $R>1$:
\begin{eqnarray*}
&&\left|U_0(s_++iv)-U_0(s_+)\right|=\left|\int _0^\infty z^{s_+-1}u_0(z) (z^{iv}-1)dz\right|\\
&&\le \int _{\frac {1} {R}}^R\left| z^{s_+-1}u_0(z) (z^{iv}-1)\right|dz+\int _R^\infty \left|z^{s_+-1}u_0(z) (z^{iv}-1)\right|dz+\\
&&\hskip 5.4cm +\int _0^{\frac {1} {R}} \left|z^{s_+-1}u_0(z) (z^{iv}-1)\right|dz\\
&&\le \int  _{ \frac {1} {R} }^R \left|z^{s_+-1}u_0(z) (z^{iv}-1)\right|dz+2\int _R^\infty \left|z^{q_0-\delta-1}u_0(z)\right|dz+\\
&&\hskip 5.4cm +2\int _0^{\frac {1} {R}} \left|z^{p_0+\delta-1}u_0(z)\right|dz.
\end{eqnarray*}
Since $z^{q_0-\delta-1}u_0, \, z^{p_0+\delta-1}u_0\in L^1(\R^+)$,  we have $\int _R^\infty \left|z^{q_0-\delta-1}u_0(z)\right|dz\to 0$ and 
$\int _0^{\frac {1} {R}} \left|z^{p_0+\delta-1}u_0(z)\right|dz\to 0$ as $R\to \infty$.

On the other hand, for all $R>1>$ fixed,
\begin{eqnarray}
\int _{\frac {1} {R}}^R \left|z^{s_+-1}u_0(z) (z^{iv}-1)\right|dz=\int _{\frac {1} {R}}^R \left|z^{s_+-1}u_0(z) (e^{iv \log z}-1)\right|dz.
\end{eqnarray}
We use now that for $s\in \R$:
\begin{eqnarray*}
|e^{is}-1|&=&\left|\int _0^se^{ir}dr\right|= \left|s \int _0^1e^{i s \rho }d\rho \right|\le |s|\\
|e^{iv \log z}-1|&\le& |v \log z|\le  (\varepsilon  \log R),\,\,\,\forall z\in \left( \frac {1} {R}, R\right),
\end{eqnarray*}
and then
\begin{eqnarray*}
\int _{\frac {1} {R}}^R \left|z^{s_+-1}u_0(z) (z^{iv}-1)\right|dz&\le& (\varepsilon  \log R) \int _0^\infty \left|z^{s_+-1}u_0(z) \right|dz\\
&\le & (\varepsilon  \log R)\int _0^\infty \left|(z^{p_0+\delta-1}+z^{q_0-\delta-1})u_0(z) \right|dz.
\end{eqnarray*}
We deduce 
$$
|\omega (\varepsilon )|\le C(p_0+\delta,q_0-\delta)\left( \int _R^\infty \left|z^{q_0-\delta-1}u_0(z)\right|dz+\int _0^{\frac {1} {R}} \left|z^{p_0+\delta-1}u_0(z)\right|dz+\varepsilon  \log R\right)
$$
and (\ref{theorem3:omega}) follows choosing $R=R(\varepsilon )$ such that $R(\varepsilon )\to \infty$ and $\varepsilon \log R(\varepsilon )\to 0$  as $\varepsilon \to 0$.

We now estimate the second term of the right hand side in (\ref{wphi:split}). By Lemma~(\ref{statphase1}) and (\ref{intU0}):
\begin{eqnarray}
\left|\frac {1} {2\pi i}\int  _{\Rea (s)=s_+, |\Ima (s)|\ge \varepsilon }\hskip -1cm U_0(s)e^{\phi (s, t, x)}ds\right|&\le &
\frac {e^{\phi (s_+(t, x), t, x)-\gamma (\varepsilon )t}} {2\pi }
\int  _{\Rea (s)=s_+, |\Ima (s)|\ge \varepsilon }\hskip -1cm\left|U_0(s)\right|ds \nonumber \\
&\le &\frac {e^{\phi (s_+(t, x), t, x)-\gamma (\varepsilon )t}} {2\pi } C(p_0+\delta, q_0-\delta). \label{secondterm}
\end{eqnarray}
This ends the proof of (a).

We prove now the point (b).
To this end we split now the integral in (\ref{def:wphi}) as follows:
\begin{eqnarray}
&&w(t, x)  
 = \frac {1} {2\pi i}\sum  _{ k\in \Z }\int  _{ s_k(t, x) -i\varepsilon  }^{ s_k(t, x) +i\varepsilon  }U_0(s)e^{\phi (s, t, x)}ds+
\frac {1} {2\pi i}\int  _{\Gamma  _{ \varepsilon  }}U_0(s)e^{\phi (s, t, x)}ds \label{wphi:split2}\\
&&\Gamma  _{ \varepsilon  }=\left\{s_+ +iv;\,\left|v-v_k \right|\ge\varepsilon \,\,\,\forall k\in \Z\right\},\,\,v_k=kv_*,
\label{wphi:Gamma}
\end{eqnarray}
where $\varepsilon >0$ is small enough to have that the intervals $\left[s_k(t, x) -i\varepsilon , s_k(t, x) +i\varepsilon \right]$ are disjoints. We first estimate the integral over $\Gamma  _{ \varepsilon  }$ using Lemma~\ref{statphase3}:
\begin{eqnarray*}
\left|\int  _{\Gamma  _{ \varepsilon  }}U_0(s)e^{\phi (s, t, x)}ds\right| &\le &
\int  _{\Gamma  _{ \varepsilon  }}\left|U_0(s)\right|e^{\Re e\phi (s, t, x)}ds\\
&\le &e^{\phi (s_+, t, x)-\gamma t}\int  _{s_+-i\infty}^{s_+ +i\infty}\left|U_0(s)\right|ds,
\end{eqnarray*}
and we conclude as for the point~(a) above.

For 
 $v\in \left[v_k- \varepsilon ,  v_k+ \varepsilon\right]$ for some $k\in \Z$, we write $s_k=s_++iv_k$ and 

\begin{eqnarray*}
&&\frac {1} {2\pi i}\int  _{ s_k(t, x) -i\varepsilon  }^{ s_k(t, x) +i\varepsilon  }U_0(s)e^{\phi (s, t, x)}ds=
\frac {U_0(s_k)} {2\pi i}\int  _{ s_k(t, x) -i\varepsilon  }^{ s_k(t, x) +i\varepsilon  }e^{\phi (s, t, x)}ds+\\
&&\hskip 3.5cm +\frac {1} {2\pi i}\int  _{ s_k(t, x) -i\varepsilon  }^{ s_k(t, x) +i\varepsilon  }\left(U_0(s)-U_0(s_k)\right)e^{\phi (s, t, x)}ds\\
 && \hskip 6cm=I_1(k)+I_2(k).
\end{eqnarray*}
Using Lemma~(\ref{statphase2})  we may repeat the analysis of the terms $I_1$ and $I_2$ performed in the proof of the point (a) above to estimate $I_1(k)$ and $I_2(k)$ as follows:
\begin{eqnarray*}
I_1(k)&=&
\frac{U_0(s_k)}{\sqrt{2\pi }} \frac {e^{\phi (s_k(t, x), t, x)}} {\sqrt{t K''(s_k)}}
\left(1+\mathcal O \left( 
 \frac {e^{-\frac {t\left( \varepsilon^2K ''(q_0)\right)} {2}}} { \varepsilon \sqrt{t K ''(q_0)}}\right) \right),
\end{eqnarray*}

\begin{eqnarray}
|I_2(k)|&\le& \omega _k(\varepsilon )
 \frac {e^{\phi (s_k(t, x), t, x)}} {\sqrt{2\pi t K''(s_k)}}\left(1+\mathcal O \left( 
 \frac {e^{-\frac {t\left( \varepsilon^2K ''(q_0)\right)} {2}}} { \varepsilon \sqrt{t K ''(q_0)}}\right) \right),\\
 \omega _k(\varepsilon )&=&\sup _{ |v|\le \varepsilon  }\left|U_0(s_k+iv)-U_0(s_k)\right|.
\end{eqnarray}
To estimate the sum of $I_1(k),$  as we have seen in (\ref{Kegal})
$$
\phi (s_k, t, x)=-s_k\log x+tK(s_k)=\phi (s_+, t, x)-\frac {2i\pi k } {\log \theta}\log x.
$$
We deduce:
\begin{eqnarray*}
e^{\phi (s_k(t, x), t, x)}&=&e^{\phi (s_+(t, x), t, x)}e^{-\frac {2i\pi k } {\log \theta }\log x},
\end{eqnarray*}
and therefore
\begin{eqnarray}
&&\sum _{ k\in \Z }I_1(k)=\left(\frac {1} {\sqrt{2\pi }}+\mathcal O \left( 
 \frac {e^{-\frac {t\left( \varepsilon^2K ''(p_0)\right)} {2}}} { \varepsilon \sqrt{t K ''(p_0)}}\right)  \right)\frac {e^{\phi (s_+(t, x), t, x)}} {\sqrt{t K''(s_+)}} \times \nonumber \\
 &&\hskip 6.4cm \times \sum _{ k\in \Z }U_0(s_k)
 e^{-\frac {2i\pi k } {\log \theta }\log x}\label{sumI1}\\
 &&\sum _{ k\in \Z }|I_2(k)|\le \left(1+\mathcal O \left( 
 \frac {e^{-\frac { \varepsilon^2 t K ''(p_0)} {2}}} { \varepsilon \sqrt{t K ''(p_0)}}\right) \right)\frac {e^{\phi (s_+(t, x), t, x)}} {\sqrt{2\pi t K''(s_+)}}
  \sum _{ k\in \Z }|\omega _k(\varepsilon )|. \label{sumI2}
\end{eqnarray}
We define $y=\frac { \log x} {\log \theta }$, (or $x=\theta ^{y}$),  and first consider the series:
$
\sum _{ k\in \Z }U_0(s_k)e^{-2i\pi k y},
$
where
\begin{eqnarray*}
U_0(s_k)&=&\int _0^\infty u_0(x)x^{s_k-1}dx=\int _0^\infty u_0(x)e^{(s_+-1+\frac {2i\pi k} {\log \theta })\log x}dx\\
&=&\int  _{ \R }u_0\left( \theta ^y\right) \theta ^{(s_+-1)y}e^{2i \pi ky} \theta ^y \log \theta dy=
\int  _{ \R }u_0\left( \theta ^y\right) \theta ^{s_+y}e^{2i \pi ky} \log \theta dy.
\end{eqnarray*}
Since the function $g(y)=u_0\left( \theta ^y\right) \theta ^{s_+y}$ is real valued, $U_0(s_{-k})=\overline{U_0(s_k)}$, and then
$U_0(s_{-k})e^{2i\pi k y}=\overline{U_0(s_k)e^{-2i\pi k y}}$. We first  deduce, using (\ref{intU0bis}), that the series is absolutely convergent.  We may then re arrange the series to obtain:
\begin{eqnarray*}
\sum _{ k\in \Z }U_0(s_k)e^{-2i\pi k y}&=&U_0(s_0)+\sum _{ k\in \N^* }\left(U_0(s_k)e^{-2i\pi k y}+\overline{U_0(s_k)e^{-2i\pi k y}}\right)\\
&=&U_0(s_+)+2\sum _{ k\in \N^* } \Re e\left( U_0(s_k)e^{-2i\pi k y}\right),
\end{eqnarray*}
where we see that the series is real valued.

Let us now turn to the sum of $I_2(k).$ We have
\begin{eqnarray*}
\sum _{ k\in \Z }\sup _{ |v|\le \varepsilon  }\left|U_0(s_k+iv)-U_0(s_k)\right|
&\le& \sum _{ k=-L}^L\sup _{ |v|\le \varepsilon  }\left|U_0(s_k+iv)-U_0(s_k)\right|+\\
&& +\sum _{ k\in \Z, |k|>L }\sup _{ |v|\le \varepsilon  }\left(\left|U_0(s_k+iv)\right|+\left|U_0(s_k)\right|\right)\\
&\le & \sum _{ k=-L}^L\sup _{ |v|\le \varepsilon  }\left|U_0(s_k+iv)-U_0(s_k)\right|+\\
&& +2C_1(p_0+\delta, q_0-\delta)\!\!\!\!\sum _{ k\in \Z, |k|>L }\frac {1} {|v_k|^2}.
\end{eqnarray*}
Since by definition $
|v_k|=\left|\frac {2\pi k} {\log \theta}\right|,$
$$
\lim _{ L\to \infty }2C_1(p_0+\delta, q_0-\delta)\!\!\!\!\sum _{ k\in \Z, |k|>L }\frac {1} {|v_k|^2}=0.
$$
Moreover, for all $L>0$ fixed,  all $k\in \Z$, $|k|\le L$ and $R>1$:
\begin{eqnarray*}
&&U_0(s_k+iv)-U_0(s_k)=\int _0^\infty z^{s_++iv_k-1} u_0(z) (z^{iv}-1)dz\\
&&|U_0(s_k+iv)-U_0(s_k)|\le \int _{\frac {1} {R}}^R \left|z^{s_+-1} z^{iv_k}u_0(z)(z^{iv}-1)\right|dz+\\
&&\hskip 4 cm +2\int _R^\infty \left|z^{q_0-\delta-1}u_0(z)\right|dz+2\int _0^{\frac {1} {R}} \left|z^{p_0+\delta-1}u_0(z)\right|dz.
\end{eqnarray*}
Arguing as above, 
$$
\lim _{ R\to \infty }\left(\int _R^\infty \left|z^{q_0-\delta-1}u_0(z)\right|dz+\int _0^{\frac {1} {R}} \left|z^{p_0+\delta-1}u_0(z)\right|dz\right)=0.
$$
For all $R>1$ fixed,
\begin{eqnarray}
\int _{\frac {1} {R}}^R \left|z^{s_+-1} z^{iv_k}u_0(z)(z^{iv}-1)\right|dz\le (\varepsilon  \log R) \int _0^\infty \left|(z^{p_0+\delta-1}+z^{q_0-\delta-1})u_0(z) \right|dz.
\end{eqnarray}
We finally have, for all $L>0$, $R\to \infty$ and $\varepsilon \to 0$ such that $t\varepsilon ^2\to \infty$:
\begin{eqnarray*}
&&\sum _{ k\in \Z }|I_2(k)|\le  \left(1+\mathcal O \left( 
 \frac {e^{-\frac { \varepsilon^2 t K ''(q_0)} {2}}} { \varepsilon \sqrt{t K ''(q_0)}}\right) \right)\frac { e^{\phi (s_+(t, x), t, x)}} {\sqrt{2\pi t K''(s_+)}}
S(L, R, \varepsilon ),\\
&&S(L, R, \varepsilon )\le 2C_1(p_0, q_0)\!\!\!\!\sum _{ k\in \Z, |k|>L }\frac {1} {|v_k|^2}+\\
&&\hskip 2cm+ 2(\varepsilon  \log R) L \int _0^\infty \left|(z^{p_0+\delta-1}+z^{q_0-\delta-1})u_0(z) \right|dz\\
&&\hskip 3cm+4L\left(\int _R^\infty \left|z^{q_0-\delta-1}u_0(z)\right|dz+\int _0^{\frac {1} {R}} \left|z^{p_0+\delta-1}u_0(z)\right|dz\right).
\end{eqnarray*}
By Lemma \ref{statphase3} and (\ref{intU0}) we may estimate now the second integral in the right hand side of (\ref{wphi:split2}) as follows:
\begin{eqnarray*}
\left|\frac {1} {2\pi i}\int  _{\Gamma  _{ \varepsilon  }} U_0(s)e^{\phi (s, t, x)}ds\right|&\le &
\frac {e^{\phi (s_+(t, x), t, x)-\gamma (\varepsilon )t}} {2\pi }\int  _{\Rea (s)=s_+, |\Ima (s)|\ge \varepsilon }\hskip -0.5cm\left|U_0(s)\right|ds \nonumber \\
&\le &\frac {e^{\phi (s_+(t, x), t, x)-\gamma (\varepsilon )t}} {2\pi } C(p_0+\delta, q_0-\delta).
\end{eqnarray*}
This proves the point (b). 
\end{proof}
\subsection{Proof of Proposition~\ref{proposition1}}
\label{subsec:prop}
\begin{proof}[Proposition~\ref{proposition1}.] 

Consider for example an initial data $u_0$ whose support lies in $(2, 3)$. For such initial data $p_0$ may be taken arbitrarily large and  $q_0$ in such a way that for all $t>0$ and $x>0$, we have $s_+(t, x)\in (p_0, q_0)$ and we may then apply Theorem~\ref{theorem3} to the solution $u$. Let then 
 $\varepsilon (t)$  be given by Theorem~\ref{theorem3}, satisfying $\varepsilon (t) \sqrt t\to \infty$

If $u$ is solution of (\ref{eq:frag})(\ref{eq:fragdata}), the function $\omega =e^tu$ is written in (\ref{wphi:split}) as the sum of two terms. The second term in the right-hand side of (\ref{wphi:split}) is proved to be like   
$e^{\Phi (s_+(t, s), t, x)}\left(1+\mathcal O\left( e^{-\gamma  _{ \delta  }(\varepsilon (t)}\right)\right)$ (cf. (\ref{secondterm})) as $t\to \infty$ and $s_+(t, x)\in (p_0+\delta , q_0-\delta )$.
The first term in the right-hand side of (\ref{wphi:split}) too is decomposed in a sum of two integrals denoted  $I_1+I_2$  given in (\ref{wphi:splitsplit}). It is straightforward to see that the error term in the estimate of $I_1(t, x)$ given in (\ref{est:I1}) is exponentially decaying as $t\to \infty$,  uniformlly for all $x$ such that $s_+(t, x)\in (p_0, q_0)$ since $\varepsilon ^2(t)t\to \infty$ as $t\to \infty$. We are then left with the term $I_2$ that we rewrite as follows:
\begin{eqnarray*}
I_2
=\frac {1} {2\pi i}\int  _{ -\varepsilon  }^{\varepsilon  }\left(U_0(s_++iv)-U_0(s_+)\right)e^{\phi (s_++iv, t, x)}dv
\end{eqnarray*}
where
\begin{eqnarray*}
&&\Phi (s, t, x)=-s\log x+tK(s)=-s\log x+t\int _0^1k_0(y)y^{s-1}dx\\
&&\Phi (s_++iv, t, x)=-(s_++iv)\log x+t\int _0^1k_0(y)y^{s_+-1}e^{iv\log y}dy\\
&&\Phi (s_++iv, t, x)=-s_+\log x+t\int _0^1k_0(y)y^{s_+-1}\cos(v\log y)dy-\\
&&-iv\log x+it\int _0^1k_0(y)y^{s_+-1}\sin(v\log y)dy\\
&&\Phi (s_+, t, x)=-s_+\log x+t\int _0^1k_0(y)y^{s_+-1}dy.
\end{eqnarray*}
Then
\begin{eqnarray*}
&&\Phi (s_++iv, t, x)=\Phi (s_+, t, x)+t\int _0^1k_0(y)y^{s_+-1}(\cos(v\log y)-1)dy+\\
&&+i\left(-v\log x+t\int _0^1k_0(y)y^{s_+-1}\sin(v\log y)dy\right)
\end{eqnarray*}
and
\begin{eqnarray*}
e^{\Phi (s_++iv)}&=&e^{\Phi (s_+)} \left(\cos\left(-v\log x+t\int _0^1k_0(y)x^{s_+-1}\sin(v\log y)dy\right)\right.+\\
&&\hskip 1cm \left. +i\sin\left(-v\log x+t\int _0^1k_0(y)y^{s_+-1}\sin(v\log y)dy\right)\right).
\end{eqnarray*}
If we write:
\begin{eqnarray*}
&&\left(U_0(s_++iv)-U_0(s_+)\right)e^{\phi (s_++iv, t, x)}=e^{\Phi (s_+)}(A+iB)\\
\end{eqnarray*}
with $A\in \R$ and $B\in \R$, then
\begin{eqnarray*}
&&\left|\int  _{ -\varepsilon  }^\varepsilon \left(U_0(s_++iv)-U_0(s_+)\right)e^{\phi (s_++iv, t, x)}dv\right|=\\
&& \left|\int  _{ -\varepsilon  }^\varepsilon e^{\Phi (s_+)} Adv+i\int  _{ \varepsilon  }^\varepsilon e^{\Phi (s_+)}Bdv \right|
\ge e^{\Phi (s_+)}\left|\int  _{ -\varepsilon  }^\varepsilon  Adv\right| .
\end{eqnarray*}
Since
\begin{eqnarray*}
U_0(s_++iv)-U_0(s_+)& =&\int _{2}^3 u_0(y) y^{s_+-1}(\cos (v\log y)-1)dy+\\
&&\hskip 2cm +
i\int _{2}^3 u_0(y) y^{s_+-1}\sin (v\log y)dy,
\end{eqnarray*}
the function $A$ is
\begin{eqnarray*}
&&A(v, t, x)=\int _{2}^3 u_0(y) y^{s_+-1}(\cos (v\log y)-1)dy\times \\
&&\hskip 3cm \times \cos\left(-v\log x+t\int _0^1k_0(y)y^{s_+-1}\sin(v\log y)dy\right)\\
&&-\int _{2}^3 u_0(y) y^{s_+-1}\sin (v\log y)dy\,\sin\!\left(-v\log x+t\!\int _0^1\!k_0(y)y^{s_+-1}\sin(v\log y)dy\right).
\end{eqnarray*}
If $|v|\le \varepsilon $ and $\varepsilon \to 0$,  we may approximate the different terms in $A$, first those depending on the initial data $u_0$,
\begin{eqnarray*}
&&\int _{2}^3 u_0(y) y^{s_+-1}(\cos (v\log y)-1)dy\approx -\frac {v^2} {2}\int _{2}^3 u_0(y) y^{s_+-1}(\log y)^2dy
=-\frac {v^2} {2}\kappa _1\\
&&\int _{2}^3 u_0(y) y^{s_+-1}\sin (v\log y)dy\approx  \int _{2}^3 u_0(y) y^{s_+-1}\left((v\log y)-\frac {(v\log y)^3} {6}\right)dy\\
&&\hskip 4.4cm \approx v\kappa _2,
\end{eqnarray*}
and those depending on the kernel $k_0$
$$\int _0^1k_0(y)y^{s_+-1}(\cos(v\log y)-1)dy\approx -\frac {v^2} {2}\int _0^1k_0(y)y^{s_+-1}(\log y)^2dy= -v^2\kappa _3,$$
\begin{eqnarray*}
\int _0^1k_0(y)y^{s_+-1}\sin(v\log y)dy&=&\int _0^1k_0(y)y^{s_+-1}\left(v\log y-\frac {v^3\log ^3 x} {6}\right)dy\\
&=&v\kappa _4-v^3\kappa _5.
\end{eqnarray*}
Notice that: 
$$\kappa _1>0,\,\,\kappa _2>0,\,\,\,\kappa _3=K''(s_+)>0,\,\,\, \kappa _4=K'(s_+)<0,\,\,\,\kappa _5=K'''(s_+)<0.$$
We consider now a curve $x=x(t)$ as follows
$$
\frac {-\log x} {t}=c
$$
for a constant $c>0$ to be chosen. Along this curve the following holds:
\begin{eqnarray*}
s_+(t, x)=(K')^{-1}\left(\frac {\log x} {t} \right)=(K')^{-1}\left(-c\right)\\
\Phi (s_+(t, x))=-s_+(t, x)\log x+tK(s_+)\\
= t\left(c(K')^{-1}(-c)+K\left((K')^{-1}(-c) \right)\right).
\end{eqnarray*}
We chose then  $c$  (it is not difficult to find kernels $k_0$ for which this is possible) such that:
$$
c(K')^{-1}(-c)+K\left((K')^{-1}(-c) \right)=0,
$$
then, along the curve $-\log x=ct$,
$$
\Phi (s_+(t, x))=0.
$$
Moreover, along that same curve,
\begin{eqnarray*}
-v\log x+\int _0^1 k_0(y)y^{s_+-1}\sin(v\log y)dy&\approx &-v\log x +v \int _0^1k_0(y)y^{s_+-1}(\log y)dy\\
&=&ctv + v\kappa _4t-v^3\kappa _5t=-v^3\kappa _5t.
\end{eqnarray*}
Then, if we chose $\varepsilon(t) $ such that $t\, \varepsilon^3 (t)\to 0$ as $t\to \infty$:

\begin{eqnarray*}
&&\cos\left(-v\log x+t\int _0^1k_0(y)y^{s_+-1}\sin(v\log y)dy\right)\approx \cos(-v^3\kappa _5t)\approx1\\
&&\sin\left(-v\log x+t\int _0^1k_0(y)y^{s_+-1}\sin(v\log y)dy\right)\approx \sin(-v^3\kappa _5t)\approx -v^3\kappa _5t\\
&&A\approx -\frac {v^2} {2}\kappa _1+\kappa _2\kappa _5v^4 t.
\end{eqnarray*}

\begin{eqnarray*}
\int  _{ -\varepsilon  }^\varepsilon e^{\Phi (s_+)} Adv\approx -\frac {\kappa _1} {2}\int  _{ -\varepsilon  }^\varepsilon v^2 dv+\kappa _2\kappa _5 t\int  _{ -\varepsilon  }^\varepsilon v^4 dv\\
=-\frac {\kappa _1} {2}\frac {\varepsilon ^3(t)} {3}+\kappa _2\kappa _5 t\frac {\varepsilon ^5(t)} {5}.
\end{eqnarray*}
This yields 
$$
|I_2(t, x)|\ge \frac {\kappa _1} {2}\frac {\varepsilon ^3(t)} {3}-\kappa _2\kappa _5 t\frac {\varepsilon ^5(t)} {5}
$$
and the proposition follows.
\end{proof}

\subsection{Proof of Corollary~\ref{cor:process}}
\label{subsec:process} 
In this section we  prove Corollary~\ref{cor:process}.  As indicated in the introduction, this result may essentially be obtained from Theorem~1 of \cite{MR2017852},  in terms of random measures, using probabilistic methods for fragmentation processes. 
\\
Let us recall that fragmentation equations may be seen as deterministic linear rate equations that describe the mass distribution of the particles involved in a fragmentation phenomenon when such phenomenon is described by a ``fragmentation'' stochastic process (see \cite{MR2253162}).  The homogeneous self-similar fragmentation process, whose associated deterministic rate equation for the mass distribution of particles is the fragmentation equation~(\ref{eq:Ffrag})(\ref{def:probak})  with $\gamma =0$, has been studied in  \cite{MR2017852}. 
\\
The aim is to describe a system of particles that split independently of each
other to give smaller particles and each obtained particle splits in turn, independent of
the past and of other particles etc.  To this end a stochastic process, $X=(X(t), t\ge0)$, is introduced that takes values in the state space denoted as  ${\mathcal{S}^\downarrow}(y)$, the set of all
sequences $Y= (y_i)_{i\in\N^*}$ such that the following holds:
\begin{equation*}
\label{condY}
  y_1 \ge ... \ge y_i \ge y_{i+1} \ge ... \ge 0
\quad\hbox{and}\quad
y = \sum_{i=1}^\infty y_i\le 1.
\end{equation*}
Under suitable assumptions on the splitting measure (giving the rate at which a particle
with mass one splits)  it is then proved in Theorem  1 (i) of \cite{MR2017852} that the  random measure $\rho_t(dy)$ defined by
$$\rho_t(dy)=\sum\limits_{i=1}^\infty X_i(t) \delta_{\f{1}{t} \log X_i(t)} (dy)$$
converges to $\delta  _{ -\mu  }$ in probability  for some $\mu <\infty$.  Moreover,  if 
$\tilde\rho_t$ is the random measure defined as the image of $\rho_t$ by the map $x\to \sqrt{t} (x+\mu)/\sigma$,  it is proved in Theorem 1 (ii) that $\tilde\rho_t$ converges in probability to the standard normal distribution ${\mathcal  N}(0,1)$. We give now the proof of  the corresponding result in terms of the density function $u$ using  Theorems~\ref{theorem1}, ~\ref{theorem2} and~\ref{theorem3}.
\\

\begin{proof}[Corollary~\ref{cor:process}.] We only write the proof of the convergence for $r(t, y)$ since the convergence of $\tilde r(t, y)$ follows from the same arguments. For any continuous and bounded  test function $\phi (y)$ with $y\in \R$:
$$\int\limits_{-\infty}^{+\infty} \phi(y)r(t,y)dy=\int\limits_{-\infty}^{K'(p_0)} +\int\limits_{K'(p_0)}^{K'(p_0+\delta)} 
+\int\limits_{K'(p_0+\delta)}^{K'(q_0-\delta)} 
+\int\limits_{K'(q_0-\delta)}^{K'(q_0)} +
\int\limits_{K'(q_0)}^{+\infty}  \phi(y)r(t,y)dy.$$
 For the first and the last integrals, are estimated using  Theorem~\ref{theorem2} to prove that they vanish.  Consider for  example the first. Using property (i) of Theorem~\ref{theorem2} we deduce that there exists $T>0$ such that, for all $t\ge T$:
 \begin{eqnarray}
&& \int\limits_{-\infty}^{K'(p_0)}  \phi(y)r(t,y)dy=\int _0^{e^{K'(p_0)t}}xu(t, x)\phi\left(\f{\log(x)}{t}\right)dx\nonumber\\
&&\qquad \le 2 ||\phi|| _{ \infty }|b_0|e^{K(p_0)t-t}\int _0^{e^{K'(p_0)t}}x^{1-p_0}dx
 =\frac{2 ||\phi|| _{ \infty }|b_0|}{2-p_0}e^{H(p_0)t},\label{S4E100}
 \end{eqnarray}
 where
 $$
 H(\nu)=K(\nu)+(2-\nu)K'(\nu)-1.
 $$
By the hypothesis (\ref{def:probak}), $H(2)=0$.  Moreover, since $H'(\nu)=(2-\nu) K''(\nu)$ and $K''(\nu)\ge 0$ we deduce that $H$ has a maximum at $\nu=2$ and therefore $H(\nu)<0$ for all $\nu\not =2$. In particular, since $p_0<1$: $H(p_0)<0$. The same argument gives:
 \begin{eqnarray}
 \int\limits_{K'(q_0)}^\infty  \phi(y)r(t,y)dy\le \frac{2 ||\phi|| _{ \infty }|a_0|}{q_0-2}e^{H(q_0)t}.\label{S4E101}
 \end{eqnarray}
 
For the third integral, we use Theorem~\ref{theorem3}-a), which gives, for $t\to  +\infty$, $\varepsilon \to 0$,  $t \varepsilon ^2\to \infty$:
\begin{eqnarray}
 \int\limits_{K'(p_0+\delta)}^{K'(q_0-\delta)}  \phi(y)r(t,y)dy=t \,\Theta_1(t, \varepsilon )\int\limits_{K'(p_0+\delta)}^{K'(q_0-\delta)}\phi(y) e^{\Psi(y) t} \f{U_0\big(K'^{-1}(y)\big) dy}{\sqrt{2\pi tK''\big(K'^{-1}(y)\big)}}\label{S4E102}
\end{eqnarray}
with $\,\Theta_1(t, \varepsilon )$ defined in \eqref{theorem3:2} and 
$$\Psi(y)=2y-yK'^{-1}(y)+K\big(K'^{-1}(y)\big)-1.$$
It is easy to check that $\Psi $ has a unique maximum at $y_0=K'(2)$ with $\Psi (y_0)=0$ and $\Psi'' (y)=-\f{1}{K''\big(K'^{-1}(y)\big)}.$
We may then  apply  Laplace's method in the right hand side of \ref{S4E102} to obtain:
 
\begin{equation}
\lim _{ t\to \infty }\int\limits_{K'(p_0+\delta)}^{K'(q_0-\delta)}  \phi(y)r(t,y)dy
=U_0(2) \phi(K'(2))\label{S4E104}
\end{equation}
with $U_0(2)=\int x u_0(x)dx$ the initial mass. 

To bound the second integral (the fourth is similar), we go back to the expression in $x$:
\begin{eqnarray*}
\int\limits_{K'(p_0)}^{K'(p_0+\delta)} \phi(y) r(t,y)dy=\int\limits_{e^{K'(p_0)t}}^{e^{K'(p_0+\delta)t}} xu(t,x)\phi\left(\f{\log(x)}{t}\right)dx\\
\leq \Vert \phi\Vert_{L^\infty(\R)} \int\limits_0^{e^{K'(p_0+\delta)t}} xu(t,x)dx.
\end{eqnarray*}
We now use  Formula~\eqref{def:invmellin} to write for $\nu \in (p_0,q_0)$
$$\left\vert u(t, x) \right\vert =\left| \frac {1} {2\pi i}\int \limits _{ \nu-i\infty }^{ \nu+i\infty }U_0(s)\,e^{\phi(s,t,x)-t}ds\right| \leq \f{1}{2\pi}  e^{{\Rea(\phi(\nu,t,x))-t}}\int\limits _{ \nu-i\infty }^{ \nu+i\infty } \vert U_0(s) \vert ds.$$
We have by definition $\phi(\nu,t,x)=-\nu \log(x) + t K(\nu).$ We thus have for $\nu\in (p_0, q_0)$
\begin{eqnarray*}
\int\limits_0^{e^{K'(p_0+\delta)t}} xu(t,x)dx \leq C e^{tK(\nu)-t} \int\limits_0^{e^{K'(p_0+\delta)t}} x e^{-\nu \log(x)}dx
=C e^{tK(\nu)-t} e^{(2-\nu)K'(p_0+\delta)t}. 
\end{eqnarray*}
Consider now the function
$$
F(\nu)=K(\nu)+(2-\nu)K'(p_0+\delta )-1.
$$
By the hypothesis~(\ref{def:probak}), $F(2)=0$. Moreover:
$$
F'(\nu)=K'(\nu)-K'(p_0+\delta )=\int _0^1k_0(x)\log x\left(x^{\nu-1}-x^{p_0+\delta-1} \right)dx
$$
and then, if $\delta >0$ is such that $p_0+\delta-1<1$:
$$
F'(2)=\int _0^1k_0(x)\log x\left(x-x^{p_0+\delta-1} \right)dx>0.
$$
There exists therefore  $\nu_1\in (p_0, 2)$ such that $F(\nu_1)<0$ and 
\begin{equation}
\int\limits_0^{e^{K'(p_0+\delta)t}} xu(t,x)dx \leq C e^{F(\nu_1)t}\to 0,\,\,\,\hbox{as}\,\,t\to \infty.\label{S4E1}
\end{equation}
In order to bound the fourth integral we write:
\begin{eqnarray*}
\int\limits_{K'(q_0-\delta)}^{K'(q_0)}  \phi(y)r(t,y)dy=\int\limits_{e^{K'(q_0-\delta )t}}^{e^{K'(q_0)t}} xu(t,x)\phi\left(\f{\log(x)}{t}\right)dx\\
\le \Vert \phi\Vert_{L^\infty(\R)} \int\limits_{e^{K'(q_0-\delta )t}}^{e^{K'(q_0)t}} xu(t,x)dx.
\end{eqnarray*}
Using again Formula~(\ref{def:invmellin}) we deduce,  for $\nu \in (2,q_0)$:
\begin{eqnarray*}
\int\limits_{e^{K'(q_0-\delta )t}}^{e^{K'(q_0)t}} xu(t,x)dx &\leq& C e^{tK(\nu)-t} \int\limits_{e^{K'(q_0-\delta )t}}^{e^{K'(q_0)t}} x e^{-\nu \log(x)}dx \nonumber\\
&=&\frac {C} {\nu-2} e^{tK(\nu)-t} \left(e^{(2-\nu)K'(q_0 -\delta )t} - e^{(2-\nu)K'(q_0)t}\right) \\
&\leq &\frac {C} {\nu-2} e^{tK(\nu)-t} e^{(2-\nu)K'(q_0 -\delta)t}.
\end{eqnarray*}
We now consider the function
$$
G(\nu)=K(\nu)+(2-\nu)K'(q_0-\delta)-1.
$$
By the hypothesis (\ref{def:probak}), $G(2)=0$ and since $K'$ is increasing,
$
G'(\nu)=K'(\nu)-K'(q_0 -\delta )
$ is positive for $\nu >q_0-\delta$ and negative for $\nu < q_-\delta,$ so if $q_0 -\delta >2$ we have $G(\nu)<0$  for any $\nu \in (2,q_0-\delta).$ Taking simply $\nu=q_0-\delta$ we have
\begin{eqnarray}
\int\limits_{e^{K'(q_0-\delta )t}}^{e^{K'(q_0)t}} xu(t,x)dx
\le \frac {C} {q_0-\delta-2} e^{tK(q_0-\delta)-t}e^{(2-q_0+\delta)K'(q_0-\delta)t}=\frac {C e^{G(q_0-\delta)t}} {q_0-\delta-2} \label{S4E2},
\end{eqnarray}
and the right-hand side vanishes exponentially when $t\to\infty.$
\end{proof}

\begin{remark}
Notice that (\ref{S4E100}), (\ref{S4E101}), (\ref{S4E1}) and (\ref{S4E2}) give rates of convergence to zero as $t\to \infty$. A careful application of  Laplace's method to the third integral would also give a rate of convergence in the limit (\ref{S4E104}) under suitable assumptions on the function
$$
h _{ \phi  }(y)=\phi(y) \f{U_0\big(K'^{-1}(y)\big) dy}{\sqrt{2\pi K''\big(K'^{-1}(y)\big)}}
$$
in a neighborhood of $y_0$. For example, if for some $\alpha >0$: 
$$
h _{ \phi  }(y)=h _{ \phi  }(y_0)+\mathcal O((y-y_0)^\alpha)\,\,\hbox{as}\,\,\,\, y\to y_0,
$$
then
$$
\int\limits_{K'(p_0+\delta)}^{K'(q_0-\delta)}  \phi(y)r(t,y)dy=U_0(2) \phi(K'(2))+\mathcal O (t^{-\frac {\alpha+1 } {2}}+t^{-1})\,\,\,\hbox{as}\,\,t \to \infty.
$$
\end{remark}
\newpage
\section{Further remarks on the asymptotic results}
\label{sec:asymptot}
\subsection{Study of the different regions of asymptotic behaviour.}
It is proved in Theorem \ref{thm: toDirac} that if $u$ is a solution of the fragmentation equation~(\ref{eq:frag}) with suitable integrability properties, $xu(t, x)$ converges to a Dirac mass at the origin as $t\to +\infty$. The formation of that Dirac mass may be followed using the descriptions of the long-time behaviour of $u$ in Theorem \ref{theorem1} and Theorem \ref{theorem2}. As already said,  Theorem \ref{theorem1} implies a uniform time  exponential decay  for $x>1$. 
 
In this section we give some indications about the domain  where, following  Theorem \ref{theorem2},  the function $xu(t, x)$ concentrates towards a Dirac measure, as $t$ increases.

All the cases of Theorem~\ref{theorem2} may be described, roughly speaking, as giving  a formula for $xu(t,x)$ of the form
\begin{equation}\label{eq:asympto}
x u(t,x) \approx A(t,\sigma) x^{-\sigma+1} e^{(K (\sigma)-1)t},
\end{equation}
with $A=a_0$ and $\sigma=q_0$ in Theorem~\ref{theorem1} and Theorem~\ref{theorem2} part $(ii),$ $A=b_0$ and $\sigma=p_0$ in Theorem~\ref{theorem2} $(i),$ and $\sigma=s_+(t,x)$, $A(t,s_+)=\frac{U_0 (s_+)}{\sqrt{2\pi t K '' (s_+)}}$ in Theorem~\ref{theorem2} part $(iii).$

Since $A$ is either a constant or  a function bounded by a power law of $t$,  the long-time behaviour is essentially given by the (exponential) behaviour of the term $x^{-\sigma+1} e^{(K (\sigma)-1)t}.$

In order to understand the behaviour of the function $xu(t, x)$ as $t$ increases we may  consider the curves $\mathcal C_r$ in the plane $(x, t)$ defined by $s_+(t, x)=r $ for any constant $r \in \R$. The curve $\mathcal C_r$ also correspond to the set of the points $(x, t)$ where $x=e^{K'(r)t}$. Since moreover, the function $K'$ is negative and increasing, it is then easy to understand the behaviour of the function $x^{-\sigma+1} e^{(K (\sigma)-1)t}$ along such curves.

Let us begin by the case of Theorem~\ref{theorem3} when $ p_0< s_+ <q_0$: we may rewrite~\eqref{eq:asympto} as 
$$x u(t,x)\approx A(t,s_+) e^{t F(s_+(t,x))}$$
where we define $F$ by
\begin{equation}\label{def:F}
F(s)={ K}( s)-1-(s-1) K'\left(s\right).
\end{equation}
Let us now turn to the cases $s_+>q_0$ and $s_+<p_0.$ Equation~\eqref{eq:asympto} becomes 
$$x u(t,x) \approx A(p) e^{t G(p,s_+)},$$
with $p=p_0$ or $p=q_0$ and $G$ defined by
\begin{equation}\label{def:G}
G(p,s)=K (p) - 1 - (p-1) K ' (s).
\end{equation}
The behaviour of the signs of the functions $G$ and $F$, which determine the exponential growth or decay of $xu(t,x)$ along the lines $x=e^{K'(s_+)t},$ is given in the following lemma. Its proof is given in the appendix (Lemma~\ref{lem:FG2}).
\begin{lemma}\label{lem:FG}
The function $F(p)$ defined by~\eqref{def:F} has two zeros $\bar p \in (p_1,1)$ and $\bar q \in (2,\infty)$. It is negative on $(p_1,\bar p)\cap (\bar q,\infty)$ and positive on $(\bar p,\bar q).$

If $p_0 <\bar p,$ the function $G(p_0,s)$ defined by~\eqref{def:G} is negative for $s\in (p_1, p_0).$ If $\bar p < p_0,$ the function $G(p_0,\cdot)$ has a unique zero $\bar s (p_0) \in (p_1,p_0),$ so that $G(p,s)$ is negative for $s<\bar s(p_0)$ and positive for $ \bar s(p_0)<s<p_0.$

Similarly:   If $q_0 >\bar q,$ the function $G(q_0,s)$ is negative for $s\in (q_0, \infty).$ If $\bar q > q_0,$ the function $G(q_0,\cdot)$ has a unique zero $\bar s (q_0) \in (q_0,\infty),$ so that $G(q,s)$ is negative for $s>\bar s(q_0)$ and positive for $q_0<s< \bar s(q_0).$

\end{lemma}

The sign of $F$ implies the exponential growth or decay of $xu(t,x)$ on the domain where Theorem~\ref{theorem3} applies.

Let us now use this lemma to follow the behaviour of $xu(t,x)$ along the curves $\mathcal C_r $, which correspond to  $x=e^{t K ' (r=s_+)}.$

\begin{itemize} 
\item $p_0 <\bar p<1$ (resp. $2<\bar q < q_0$): $G(p_0,s_+)$ is negative for $s_+ \in (p_1,p_0)$ (resp. $G(q_0,s_+)$ is negative for $s_+\in (q_0,\infty)$), we have an exponential decay on $s_+ < p_0$ (resp. $q_0 < s_+$). For $s$ on the interval $ (p_0,2)$ (resp. $(1,q_0)$), the behaviour is given by the function $F$: exponential decay for $s_+\in (p_0,\bar p)$ (resp. $s_+\in (\bar q, q_0)$) and exponential growth for $s_+\in(\bar p, 2)$ (resp. $s_+\in (1,\bar q)).$ 

All together, this gives a limit for the  zone of exponential growth given on the left by $s_+=\bar p$ (resp. on the right by $s_+=\bar q$).

\item $\bar p <p_0<1$ (resp. $2< q_0 <\bar q $): $G(p_0,\cdot)$ has a unique zero $\bar s(p_0) \in  (p_1,p_0)$ (resp. $G(q_0,\cdot)$ has a unique zero $\bar s(q_0)$ on $(q_0,\infty)$), and we have an exponential decay for $s_+ < \bar s (p_0)$ (resp. $\bar s (q_0) < s_+$) and an exponential increase for $\bar s (p_0) < s_+ < p_0$ (resp. $q_0 < s_+<\bar s (q_0)$). For $s_+ \in (p_0, 2)$ (resp. $s_+\in (1,q_0)$), the behaviour is given by the function $F$: since $\bar p < p_0$ (resp. $q_0 < \bar q$), there is an exponential growth for $s_+\in (p_0,2)$ (resp. $s_+\in (1,q_0)$). 

All together, this gives a limit for the zone of exponential growth given on the left by $s_+ = \bar s(p_0)$ (resp. on the right by $s_+ = \bar s(q_0)$). 

\end{itemize}

Of course, any combination of these two cases, for the respective positions of $\bar p$ and $p_0$ for "$s_+$ small" and of $\bar q$ and $q_0$ for "$s_+$ large" is possible, which finally leads to four possible situations. 

We also notice that for the case of very regular initial data, \emph{i.e.} $p_0 <\bar p$ and $\bar q < q_0,$ the zone of exponential growth or decay depends only on the fragmentation kernel properties, which define the values of $\bar p$ and $\bar q,$ and endly the dependence on the initial condition only appears as a correction term (presence or not of $\f{1}{\sqrt{t}}$). 

\subsection*{Example 1:  homogeneous kernel.}
\label{sec:homogeneous}
For the homogeneous kernel, $k_0=2$ so that Equation~\eqref{def:mellin:noyau} gives ${ K}(s)=\frac{2}{s},$ so that $p_1=0$ and we have
$$s_+(t,x)=\sqrt{\f{2t}{-\log(x)}},\qquad K'(s_+)=-\f{2}{s_+^2}.$$
We calculate easily that 
$$F(s)=\f{2}{s}-1+2\f{s-1}{s^2}=\f{-s^2+ 4s-2}{s^2} =-\f{\big(s-(2+\sqrt{2})\big)\big(s-(2-\sqrt{2})\big)}{s^2},$$
so that $\bar p= 2-\sqrt{2},$ $\bar q = 2+\sqrt{2},$ and
$$G(p,s)=\f{2}{p}-1+2\f{p-1}{s^2}=\f{1}{s^2} \biggl(s^2 (\f{2}{p} -1) + 2(p-1)\biggr),$$
so that $\bar s(p_0)= \sqrt{\f{2p_0(1-p_0)}{2-p_0}},$ $\bar s (q_0)=\sqrt{\f{2q_0(q_0-1)}{q_0-2}}.$ Notice that $\bar s (p_0) \to 0=p_1$ when $p_0 \to 1,$ and $\bar s (q_0) \to \infty$ when $q_0 \to 2:$ the ``less regular" the initial data, the largest the domain of exponential growth.

\begin{figure}[ht]
\centering
\includegraphics[width=0.8\textwidth,height=0.3\textheight]{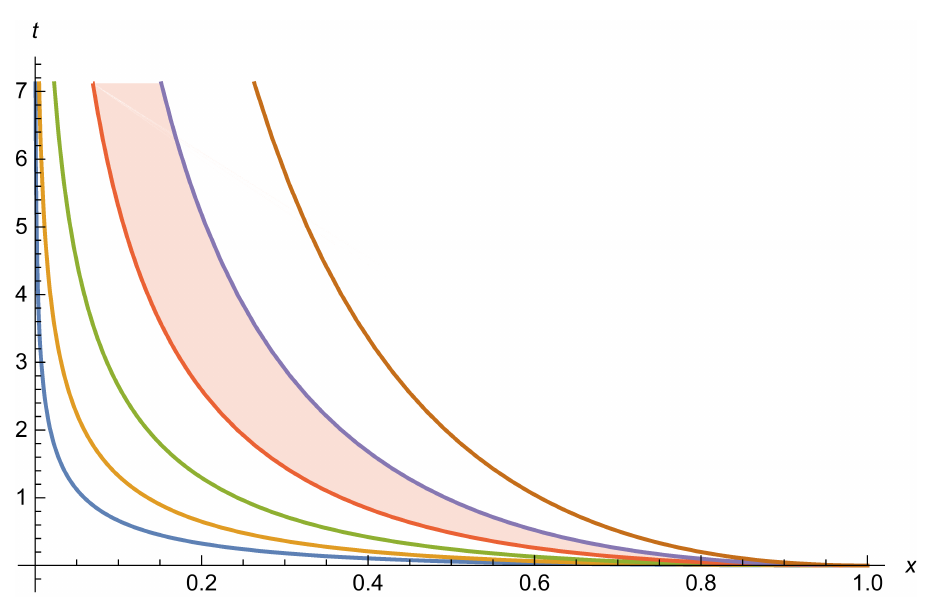}
\caption{\label{fig:splus} Different curves of the form $s_+=\gamma$ for different values of $\gamma >0$, so that $2t=-\gamma^2 \log x$.The function $xu(t, x)$ tends to zero exponentially fast uniformly out of the shaded region. It tends to infinity inside that region. As $t\to \infty$, the function $x u(t, x)$ concentrates in the interval  $x\in \displaystyle{\left( e^{-\frac {2t} {\gamma _\ell^2}}, e^{-\frac {2t} {\gamma _r^2}}\right)}.$
}
\end{figure}

The values of $\gamma _{ r } $ in the light blue curve and $\gamma  _{ \ell}$ in the green curve in Figure~\ref{fig:splus} depend on the relative values of $q_0$ and $\bar q=2+\sqrt 2$ and of $p_0$ and $\bar p=2-\sqrt 2$ respectively, as detailed above for the general case.

\subsection*{Example 2:  Mitosis kernel.} 
For the mitosis kernel, $k_0=2\delta_{z=\f{1}{2}}$: Condition~H is satisfied, and $K (s)=2^{2-s},$ $p_1=-\infty$ and we have $s_+(t,x)$defined by
$$2^{2-s_+(t,x)}  = e^{(2-s_+)\log(2)}=-\f{\log(x)}{t \log(2)},\qquad s_+(t,x)=2-\f{\log\biggl(-\f{\log(x)}{t\log(2)}\biggr)}{\log(2)}.$$
We calculate easily that
$$F(s)=2^{2-s} -1 + (s-1) \log(2) 2^{2-s},$$
and $\bar p<1$ and $\bar q>2$ are defined by $F(\bar p / \bar q)=0,$ see Figure~\ref{fig:Fmito}. 
$$G(p,s)=2^{2-p} -1 + (p-1) \log(2) 2^{2-s},$$
 so that $\bar s(p_0)=2- \f{1}{\log(2)}\log\biggl(\f{1-2^{2-p_0}}{(p_0-1)\log(2)}\biggr),$ $\bar s(q_0)=2- \f{1}{\log(2)}\log\biggl(\f{1-2^{2-q_0}}{(q_0-1)\log(2)}\biggr).$

\begin{figure}[ht]
\centering
\includegraphics[width=0.9\textwidth,height=0.3\textheight]{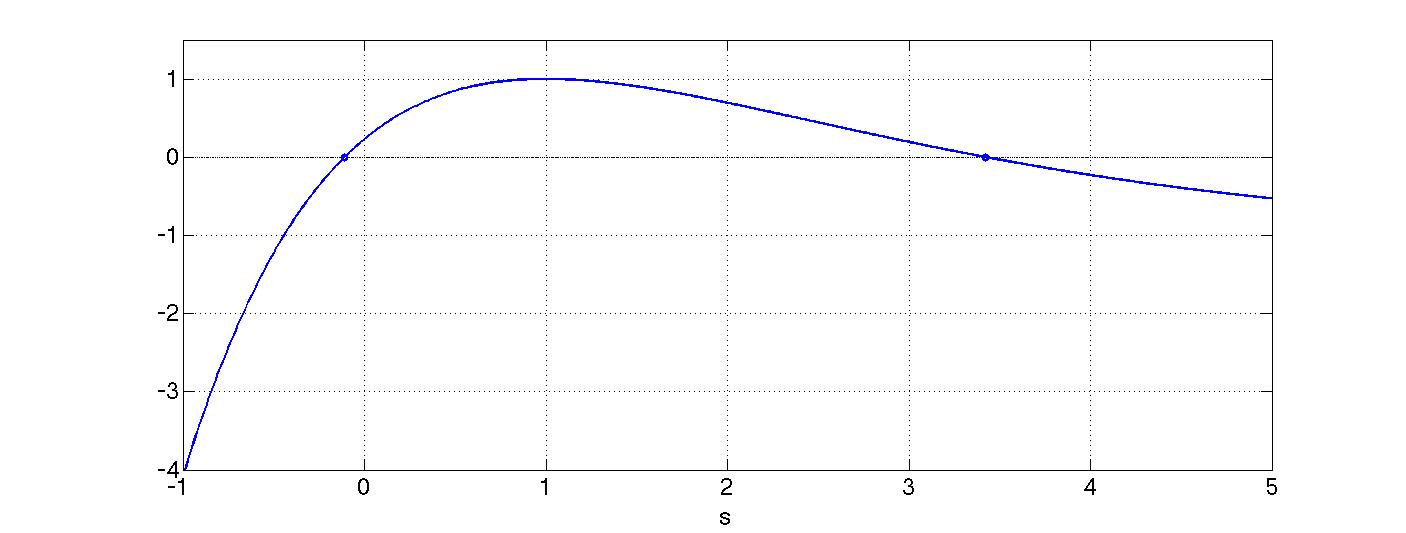}
\caption{\label{fig:Fmito}Function $F(s)$ defined by~\eqref{def:F} for $k_0=2\delta_{y=\f{1}{2}}.$ We see the two zeros $\bar p <0$ and $\bar q >3.$}
\end{figure}

\subsection{The asymptotic behaviour for the growth-fragmentation equation.}\label{sec:croisfrag}
The asymptotic behaviour for the growth-fragmentation equation~\eqref{eq:transp:frag} may also be deduced from Theorems~\ref{theorem1} and~\ref{theorem2}. 
Using Formula~\ref{lin:uv}, we know that the solution $v$ satisfies $v(t, x)=e^{-ct}u(t, xe^{-ct}):$ we can apply directly the results of  Theorems~\ref{theorem1} and~\ref{theorem2}. It remains to analyse in which part of the plane $(x,t)$ there is an exponential growth or decay.

First, for $x=\alpha e^{ct}$ with $\alpha >1,$ Theorem~\ref{theorem1} implies, for a constant $C>0,$
$$ x v(t,x) =  \alpha u(t,\alpha) \sim C \alpha^{1-q_0} e^{(K(q_0)-1)t},$$
so that there is an exponential decrease for the domain $x>e^{ct}.$ 

For the domain $x< e^{ct},$ 
the lines where we can follow the mass concentration are now given by $s_+(t,xe^{-ct})$ constant, \emph{i.e.} $x=e^{(c+K'(s_+))t}.$ Notice that contrarily to the fragmentation equation, these lines either go to infinity or to zero, with a limit case for $s_+=K'^{-1} (-c).$ 

To avoid too long and repetitive considerations, we focus on the case of very smooth data, where $p_0<s_+<q_0$ is the main domain of interest in Theorem~\ref{theorem2}.
On that domain, for a constant $C>0,$
$$ x v(t,x) =  e^{K'(s_+)t} u(t,e^{K'(s_+)t}) \sim \f{C}{\sqrt{t}} e^{(K'(s)(1-s) +K(s)-1)t},$$
where we recognize $e^{F(s)t}$ with $F$ defined by~\eqref{def:F}. 
We can thus apply directly the study done for the fragmentation equation: the domain of exponential growth is for $\bar p < s_+(t,xe^{-ct}) < \bar q,$ with $\bar p, \bar q$ defined in Lemma~\ref{lem:FG}. This corresponds to curves $x=e^{(c+K'(s_+))t}$ with $K'(s_+) \in \bigl(K'(\bar p), K'(\bar q)\bigl).$  What is new here is to investigate whether these curves go to zero or to infinity in large times. 
We also notice that the curve of maximal exponential growth is given for $F(s)=F(1)=K(1)-1.$
We have the following cases.

\begin{itemize}
\item $c >-K'( \bar p):$ the zone of exponential growth goes to infinity.
\item $-K'(\bar q) < c < -K'(\bar p):$ the domain of exponential growth covers a wide range of lines $x$, going to infinity or to zero. 
In particular, the line of maximal growth, given by $x=e^{(c+K'(1))t},$ may either go to zero if $c<-K'(1)$ or to infinity if $c>-K'(1).$
\item $ c< -K'(\bar q):$  the domain of exponential growth goes to zero.
\end{itemize}

\subsection{Numerical illustration}
To visualize the long-time behaviour of Equation~\eqref{eq:frag} as described in the previous results, we choose to simulate it in the log-variable $y=\log(x).$ The equation becomes, for $n(t,y)=u(t,x),$
\begin{equation}
\f{\p}{\p t} n(t,y) + n(t,y) = \int\limits_0^\infty  \kappa_0 ( z) n(t,y+z)dz,
\end{equation}
where $\kappa_0(z)=k_0 (e^{-z}).$ In the case of the mitosis kernel $k_0=2\delta_{x=\f{1}{2}},$ we have 
\begin{equation}
\f{\p}{\p t} n(t,y) + n(t,y) = 4 n(t,y+\log 2), \quad n(0,y)=n^{in} (y). 
\end{equation}
We choose a gaussian for the initial data $n^{in}=e^{-\f{(y-y_0)^2}{2}},$ with $y=-5,$ and following the solution in time we draw the limits of the zone where $n(t,y)\geq 0.1 \max\limits_x n(t,y).$ This is given in Figure~\ref{fig:mitose}. A linear fit of the form $t=a_1 y + a_2$ gives excellent results. 
This corresponds to curves $x=e^{\f{t-a_2}{a_1}},$ which are of the predicted shape. 
\begin{figure}[ht]
\centering
\includegraphics[width=0.95\textwidth,height=0.3\textheight]{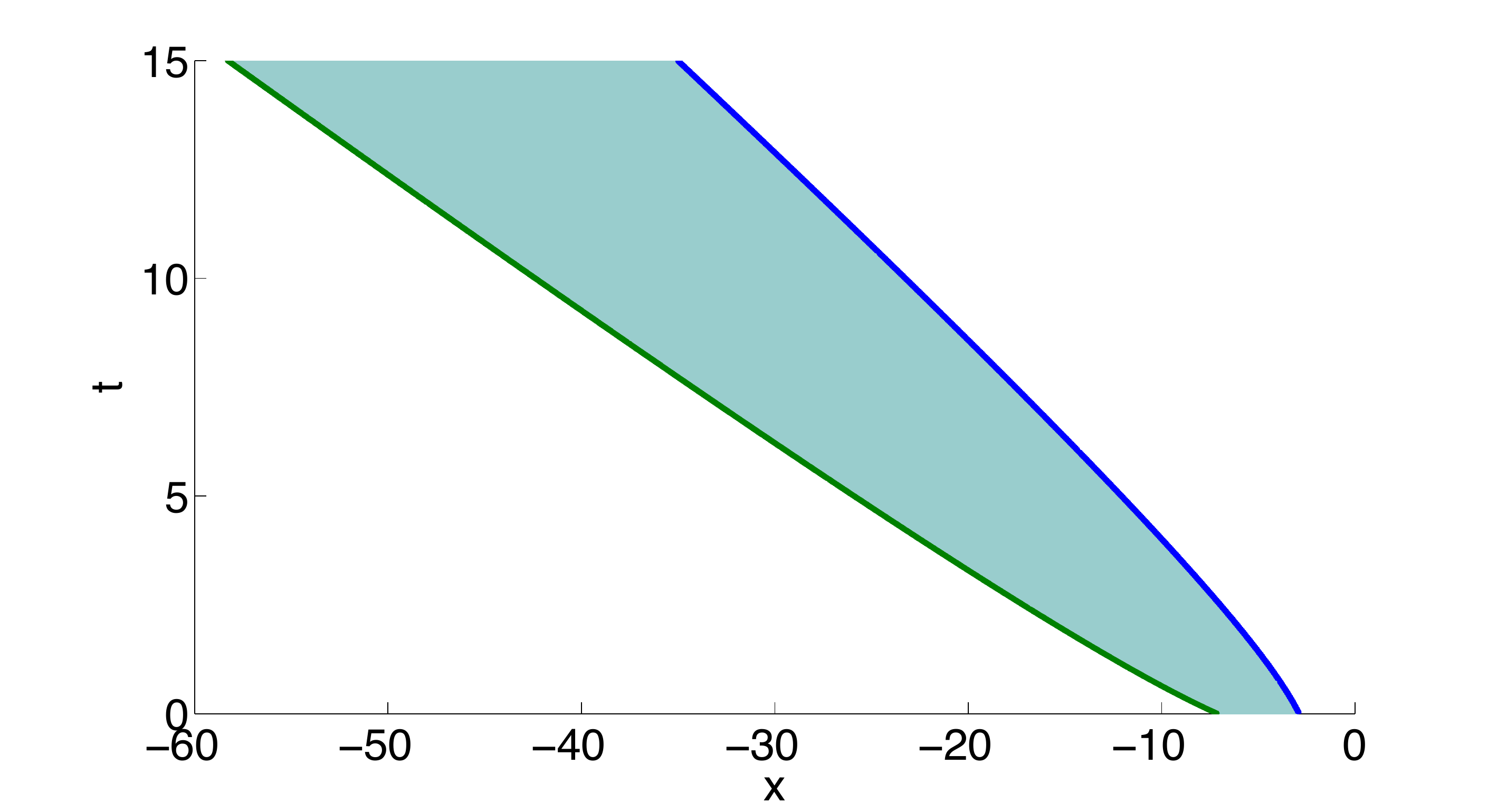}
\caption{\label{fig:mitose} The solution of the growth-fragmentation equation in a log-scale. The two curves in blue and green represent limits for the major part of the support of $u(t,x):$ inside these two curves, the solution $u(t,x)$ is larger than $10\%$ of its maximal value in $x$ at time $t$.}
\end{figure}

\section*{Appendix}
We give in this appendix the statements and proofs of some important auxiliary results.
\begin{lemma}\label{lem:phi}
Let $k_0$ be defined by~\eqref{def:probak}, $K$ its Mellin transform defined by~\eqref{def:mellin:noyau}, $p_1 \in [-\infty, 1)$ defined by~\eqref{k0:p1}, and $\phi(s,t,x)$ defined by~\eqref{def:phi2}. Then for $x<1,$  $\phi(\cdot,t,x)$ is convex, it decreases for $s \in (p_1,s_+(t,x))$ and increases for $s \in (s_+(t,x),+\infty)$ where $s_+(t,x)$ is defined by
\begin{eqnarray}
\label{alphaplus2}
s_+(t, x)=(K')^{-1}\left(\f{\log(x)}{t}\right).
\end{eqnarray}
\end{lemma}
\begin{proof}
We consider the derivative of $\phi$ with respect to  $s$
$$\f{\p\phi}{\p s}(s, t, x) = -\log(x)+t K'(s)=-\log(x)+t \int\limits_0^1 (\log(z))z^{s-1} k_0(z)dz.$$
The second derivative of $\phi$ with respect to  $s$ is
$$\f{\p^2\phi}{\p s^2}(s, t, x) = t K''(s)=t \int\limits_0^1 (\log(z))^2 z^{s-1} k_0(z)dz > 0,$$
so that $\phi(\cdot,t,x)$ is convex.
By definition,
$$
K'(\cdot): (p_1, \infty)\longrightarrow (-\infty, 0 ),\qquad K'(s)=\int\limits_0^1 (\log(z))z^{s-1} k_0(z)dz
$$
is an  increasing bijective function from $(p_1\leq 1, \infty)$ to $(-\infty, 0 )$. Then,  for each $t>0$, $x>0$:
$$\f{\p\phi}{\p s}(\cdot, t, x): (p_1, \infty)\longrightarrow (-\infty, -\log x >0)$$ 
is also increasing and bijective as a function of $s$, so that for all $x\in (0, 1)$, 
the function  $\f{\p\phi}{\p s}(\cdot, t, x)$ has a unique zero on $(p_1, \infty)$ given by~\eqref{alphaplus2}.
\end{proof}
Since  $K'$ is an increasing function of $s$,  $s_+(\cdot, x)$ is an increasing function of $t$ and $x\in (0,1)$.

\begin{lemma}[Lemma~\ref{lem:FG}]\label{lem:FG2}
The function $F(p)$ defined by~\eqref{def:F} has two zeros $\bar p \in (p_1,1)$ and $\bar q \in (2,\infty)$. It is negative on $(p_1,\bar p)\cap (\bar q,\infty)$ and positive on $(\bar p,\bar q).$

If $p_0 <\bar p,$ the function $G(p_0,s)$ defined by~\eqref{def:G} is negative for $s\in (p_1, p_0).$ If $\bar p < p_0,$ the function $G(p_0,\cdot)$ has a unique zero $\bar s (p_0) \in (p_1,p_0),$ so that $G(p,s)$ is negative for $s<\bar s(p_0)$ and positive for $ \bar s(p_0)<s<p_0.$

Similarly:   If $q_0 >\bar q,$ the function $G(q_0,s)$ is negative for $s\in (q_0, \infty).$ If $\bar q > q_0,$ the function $G(q_0,\cdot)$ has a unique zero $\bar s (q_0) \in (q_0,\infty),$ so that $G(q,s)$ is negative for $s>\bar s(q_0)$ and positive for $q_0<s< \bar s(q_0).$
\end{lemma}

\begin{proof}
We first notice that $F(p_1)=-\infty$ (this may be viewed by going back to the definition of $ K$, $p_1$ and $K'$), $F(1)=K (1)-1 >0$, $F(2)=-K' (2)>0$ and $F(+\infty)=-1,$ so that by continuity $F$ has at least two zeros $\bar p \in (p_1, 1)$ and $\bar q \in (2,\infty).$ Since $F'(s)=(1-s) K ' (s),$ $F$ is increasing on $(p_1,1)$ and decreasing on $(1,\infty):$ $\bar p$ and $\bar q$ are its only zeros. This implies that $F$ is negative on $(p_1,\bar p)\cup (\bar q, \infty)$ and positive on $(\bar p, \bar q).$ 
 
We notice that $G(s,s)=F(s).$ Considering $G$ as a function of $s$ we have 
$$\frac{\p}{\p s} G(p,s)=(1-p)K '' (s),$$
so that $\frac{\p}{\p s} G(p,s) >0$ for $p<1,$ $\frac{\p}{\p s} G(p,s)<0$ for $p>1:$ $G(p_0,\cdot)$ is an increasing function of $s$ for $p_0\in (p_1,1)$ and $G(q_0,\cdot)$ is a decreasing function of $s$ for $q_0\in (2,\infty).$ Since $G(p,p)=F(p)$ is positive on $(\bar p,\bar q)$ and negative on $(p_1,\bar p) \cup (\bar q, \infty)$, and $G(p_0<1,p_1) = - \infty,$ and $G(q_0>2,\infty)=K (q_0) -1<0,$ we have two cases for each part $(p_1,1)$ and $(2,\infty)$ according to the respective position of $\bar p$ and $p_0,$ $\bar q$ and $q_0$, as stated in the lemma.
 \end{proof}
 
 \subsection*{Weak solutions of the fragmentation equation (\ref{eq:frag}).}
 \begin{definition}
 A function $u \in L^\infty([0, \infty); L^1((1+x)dx))$ is said a weak solution of the fragmentation equation (\ref{eq:frag}) if for all $\varphi \in C_c([0, \infty))$ and a.e. $t>0$:
 \begin{eqnarray}
\int _0^\infty u(t, y)\varphi (y)dy&=&\int _0^\infty u_0(y)\varphi (y)dy+\nonumber\\
&+&\int _0^t\int _0^\infty u(s, y)
\left(\int _0^1\varphi (z y)k_0(z)dz-\varphi (y)\right)dy ds.\label{eq:wfrag}
 \end{eqnarray}
 \end{definition}
If  $u_0\in  L^1((1+x)dx)$ we may obtain a weak solution  $$u\in C([0, \infty ); L^1((1+x)dx))\cap
L^\infty ([0, \infty ); L^1(xdx))$$ of Equation~(\ref{eq:frag}) as follows.  We first obtain a solution  $w$ of the integral equation 
\begin{equation}
\label{S5IntEq}
w(t, x)=u_0+\int _0^t\int_x^\infty \f{1}{y}k_0\left(\f{x}{y}\right)w(s,y)dyds,\,\,\, a.e. t>0.
\end{equation}
To this end we define:
$$
T(w)(t, x)=u_0(x)+\int _0^t\int_x^\infty \f{1}{y}k_0\left(\f{x}{y}\right)w(s,y)dyds.
$$
We claim that  $T$ is a contraction from 
$ C([0, \tau ]; L^1((1+x)dx))$ into itself for any $0<\tau <1$. 
\begin{eqnarray*}
||T(w)(t)|| _{ L^1 }&\le &||u_0|| _{ L^1 }+\int _0^t\int _0^\infty\int_x^\infty \f{1}{y}k_0\left(\f{x}{y}\right)|w(s,y)|dydxds\\
&=&||u_0|| _{ L^1 }+\int _0^t\int _0^\infty\int_0^y \f{1}{y}k_0\left(\f{x}{y}\right)|w(t,y)|dxdyds\\
&=&||u_0|| _{ L^1 }+\int _0^t\int _0^\infty\int_0^1 k_0(z)|w(s,y)|dzdyds\\
&\le &||u_0|| _{ L^1 }+ \tau \sup _{ s\in (0, \tau ) }||w(s)|| _{ L^1 }.
\end{eqnarray*}
\begin{eqnarray*}
||T(w)(t)|| _{ L^1 (xdx)}&\le &||u_0|| _{ L^1(xdx) }+\int _0^t\int _0^\infty x\int_x^\infty \f{1}{y}k_0\left(\f{x}{y}\right)|w(s,y)|dydxds\\
&=&||u_0|| _{ L^1(xdx) }+\int _0^t\int _0^\infty \int_0^y x\f{1}{y}k_0\left(\f{x}{y}\right)|w(s,y)|dxdyds\\
&=&||u_0|| _{ L^1(xdx) }+\int _0^t\int _0^\infty\int_0^1 zk_0(z)y|w(s,y)|dzdyds\\
&\le &||u_0|| _{ L^1(xdx) }+ \tau \sup _{ s\in (0, \tau ) }||w(s)|| _{ L^1(xdx) }.
\end{eqnarray*}

If $w_1$ and $w_2$ are in $L^1((1+x)dx)$ and $w=w_1-w_2$:
\begin{eqnarray*}
||T(w_1)(t)-T(w_2)(t)|| _{ L^1 }&\le &\int _0^t\int _0^\infty\int_x^\infty \f{1}{y}k_0\left(\f{x}{y}\right)|w(s,y)|dydxds\\
&=&\int _0^t\int _0^\infty\int_0^y \f{1}{y}k_0\left(\f{x}{y}\right)|w(s,y)|dxdyds\\
&=&\int _0^t\int _0^\infty\int_0^1 k_0(z)|w(s,y)|dzdyds\\
&\le & \tau \sup _{ s\in (0, \tau ) }||w(s)|| _{ L^1 }.
\end{eqnarray*}
Therefore, there is a fixed point 
$w\in C([0, \tau]; L^1((1+x)dx))$ of $T$. Since the time of existence is independent of the initial data, this procedure may be iterated to obtain a solution $w\in C([0, \infty); L^1((1+x)dx))$ of (\ref{S5IntEq}) in $L^1((1+x)dx))$.  We immediately deduce  $w\in C([0, \infty ); L^1((1+x)dx))\cap W _{ loc }^{1, \infty}((0, s ); L^1((1+x)dx))$. Then,
the function 
$$
u(t, x)=e^{-t}w(t, x)
$$
satisfies, for a.e. $t>0$:
\begin{eqnarray*}
&&u_t+u=-e^{-t}w+e^{-t}w_t+w=e^{-t}\int_x^\infty \f{1}{y}k_0\left(\f{x}{y}\right)w(t,y)dy \nonumber \\
&&\Longleftrightarrow u_t+u=\int_x^\infty \f{1}{y}k_0\left(\f{x}{y}\right)u(t,y)dy
\end{eqnarray*}
in $L^1((1+x)dx)$. In particular, after multiplication of both sides of the equation by a test function $\varphi \in L^\infty[0, \infty)$ we obtain for a.e. $t>0$:
\begin{eqnarray*}
\int _0^\infty u(t, y)\varphi (y)dy&=&\int _0^\infty u_0(y)\varphi (y)dy\\
&+&\int _0^t\int _0^\infty u(s, y)\left(\int _0^1\varphi (z y)k_0(z)dz-\varphi (y)\right)dy ds.
 \end{eqnarray*}
We deduce that $u$ is a weak solution. If we choose $\varphi (x)=x$ we obtain for a.e. $t>0$:
\begin{eqnarray*}
\int _0^\infty y u(t, y)dy&=&\int _0^\infty y u_0(y)dy+\int _0^t\int _0^\infty  u(s, y)\left(\int _0^1zy k_0(z)dz-y\right)dy ds\\
&=&\int _0^\infty y u_0(y)dy.
 \end{eqnarray*}
So the mass of $u$ is constant for all $t>0$.

Suppose finally that we have two functions $u_i\in L^1$, $i=1, 2$,  satisfying (\ref{eq:wfrag}).
Then, if $u=u_1-u_2$:
\begin{eqnarray*}
u(t, x)=-\int _0^tu(s, x)ds+\int _0^t\int_x^\infty \f{1}{y}k_0\left(\f{x}{y}\right)u(s,y)dy\\
 \end{eqnarray*}
and,
\begin{eqnarray*}
\int _0^\infty |u(t, x)|dx&\le &\int _0^t\int_0^\infty\int_x^\infty \f{1}{y}k_0\left(\f{x}{y}\right)|u(s,y)|dydx+\int _0^t\int _0^\infty |u(s, x)|dsdx\\
&=&\int _0^t\int_0^\infty\f{1}{y}\int_0^y k_0\left(\f{x}{y}\right)|u(s,y)|dxdy+\int _0^t\int _0^\infty |u(s, x)|dsdx\\
&=&\int _0^t\int_0^\infty|u(s,y)|dy\int_0^1 k_0(z)dz+\int _0^t\int _0^\infty |u(s, x)|dsdx,
 \end{eqnarray*}
 and we deduce $\int _0^\infty |u(t, x)|dx=0$ and $u_1(t)=u_2(t)$ in $L^1$ for a. e. $t>0$. 
 
Notice finally that if we also impose to the initial data to be non negative, i.e. $u_0\ge 0$, then the weak solution $u$ is also non negative. We would first prove the existence of a non negative fixed point $\widetilde w$ of $T$ as above since this operator sends the positive cone of  $C([0, \tau ]; L^1((1+x)dx))$ into itself.  The function $\widetilde u=e^{-t}\widetilde w$ is then a non negative weak solution of (\ref{eq:wfrag}). By uniqueness of the weak solution it follows that $u=\widetilde u \ge 0$.
 
 \subsection*{On the condition H}
 \label{subsectionH}
 We present here some useful remarks on Condition H. 
 \begin{proposition}
\label {propositioncondH1}
Given a sequence $(\sigma  _{ \ell }) _{ \ell \in \N }$ in $(0,1)$, Condition~H is a necessary and sufficient condition in order to have at least one real number $v\not = 0 $ with the following property:
\begin{eqnarray}
\label{vproperty}
\forall \ell\in \N, \exists k(\ell)\in \N; \,\,\vert v\log \sigma  _{ \ell }\vert =2k(\ell)\pi.
\end{eqnarray}
\end{proposition}
\begin{proof}
Notice that $v=0$ satisfies  property~\eqref{vproperty}. If Condition $H$ is satisfied, the number 
\begin{equation}
\label{def:vstar2}
v_*=\frac {2\pi } {\log \theta} <0
\end{equation}
 has the property (\ref{vproperty}) since, for all $\ell$, $\sigma _\ell=\theta^{p_\ell}$ and then,
$$
v_*\log \sigma  _{ \ell }=v_*p_\ell\log \theta=2p_\ell \pi .
$$
Suppose on the other hand that such a $v\in \R\setminus\{0\}$ exists. Then, for all $\ell \not =0$:
$$
-\vert v\vert=\frac {2\pi k(\ell)} {\log \sigma _{ \ell } }=\frac {2\pi k(0)} {\log \sigma _{0} }
$$
and therefore
$$
\sigma _\ell=\sigma _0^{\frac {k(\ell)} {k(0)}}=\left( \sigma _0^{\frac {1} {k(0)}}\right)^{k(\ell)},
$$
and Condition H follows with $\theta=\sigma _0^{\frac {gcd(k(\ell))} {k(0)}}$ and $p_\ell=k(\ell)$.
\end{proof}

\begin{proposition}
\label {propositioncondH2}
Given a sequence $(\sigma  _{ \ell }) _{ \ell \in \N }$ of real numbers in $(0,1)$ satisfying  Condition~H and $v_*,$ $Q$ defined by (\ref{def:vstar2}), the set of real numbers $v$ that satisfy the property (\ref{vproperty}) is given by the set $Q=v_* \Z.$
\end{proposition}
\begin{proof}
Suppose that $v\not =v_*$ is a real number satisfying (\ref{vproperty}) and fix  $\ell \in \N$. Then,
\begin{equation}
\label{deuxv}
v_*\log \sigma  _{ \ell }=2p_\ell\pi,\,\,\,v\log \sigma  _{ \ell }=2k(\ell)\pi 
\end{equation}
and we must  have:
$$
v=\frac {k(\ell)} {p_\ell}v_*.
$$
Suppose therefore that 
$$
v=\frac {p} {q}v_*
$$
where $p/q$ is irreducible.  We deduce from (\ref{deuxv}) that for any $\ell \in \N$:
$$
v\log \sigma  _{ \ell }=\frac {p} {q}v_*\log \sigma  _{ \ell }=2\frac {p} {q}p_\ell \pi =2k(\ell)\pi .
$$
and  $p_\ell$ must be a multiple of $q$ for any $\ell$. Since, by assumption, $1$ is the only common divisor of all $p_\ell,$ we deduce that $q=1$ and $v\in \Z v_* = Q.$ 

If, on the other hand $v \in \Z v_*,$ the property~\eqref{vproperty} is immediate.
\end{proof}

\begin{proposition}
\label {propositioncondH3}
Suppose that $b _{ \ell }>0$ for all $\ell\in \N$ and $B=\sum_{ \ell\in \N }b _{ \ell }<\infty$ . Then,  for all $\varepsilon >0$ there exists $\delta >0$ such that:
\begin{equation*}
\label{esti:propositioncondH3}
\sup _{v,\, d(v, Q)\ge \varepsilon  }\,\,\Re e\left(\sum _{ \ell\in \N }b _{ \ell }e^{iv\log \sigma  _{ \ell }}\right)\le 
B(1-\delta).
\end{equation*}
\end{proposition}

\begin{proof} Without lack of generality, we assume  $B=1$.
By definition of $Q$, if $v\not \in Q$, there exists $\ell_0 \in \N$ such that 
$$
\frac {v\log \sigma  _{ \ell _0}} {2\pi }\not \in \Z.
$$
Then,
\begin{eqnarray*}
\Re e\left(\sum _{ \ell\in \N }b _{ \ell }e^{iv\log \sigma  _{ \ell }}\right)&=&	b _{ \ell_0 }\cos(v\log \sigma  _{ \ell_0 })+\sum _{ \ell\in \N, \ell \not = \ell_0 }b _{ \ell }\cos(v\log \sigma  _{ \ell })\\
&\le &1+b _{ \ell_0 }\left(\cos(v\log \sigma  _{ \ell_0 })-1 \right)\\
&=& 1-b _{ \ell_0 }\left(1-\cos(v\log \sigma  _{ \ell_0 })\right)
\end{eqnarray*}
Since the function $v\to \Re e\left(\sum _{ \ell\in \N }b _{ \ell }e^{iv\log \sigma  _{ \ell }}\right)$ is continuous, for all $R>0$:
$$
\sup _{ d(v, Q)\ge \varepsilon , |v|\le R }\Re e\left(\sum _{ \ell\in \N }b _{ \ell }e^{iv\log \sigma  _{ \ell }}\right)<1.
$$
Suppose now that for some sequence $v_n\to \infty$,
$$
\lim _{ n\to \infty }\sum _{ \ell\in \N }b _{ \ell }e^{iv_n\log \sigma  _{ \ell }}=1.
$$
Since $\sigma  _{ \ell }=\theta^{p _{ \ell }}$,
$$
\lim _{ n\to \infty }\sum _{ \ell\in \N }b _{ \ell }e^{iv_n p _{ \ell }\log\theta}=1.
$$
Since $\theta \in \R$ and $v_n\in \R$ for all $n$,  $(e^{iv_n \log\theta}) _{ n\in \N }$ is bounded in $\C$. Then, there is a subsequence $(v_k) _{ k\in \N }$ and $A\in \C$ such that $|A|= 1$ and
$$
e^{iv_n \log\theta}\to A,
$$
and then, for all $\ell \in \N$:
$$
e^{iv_n p _{ \ell }\log\theta}\to A^{p _{ \ell }}.
$$
On the other hand, since $v_n\in \R$, $p _{ \ell }\in \N$ and  $\theta\in \R$:
$$
\left|b _{ \ell }e^{iv_n p _{ \ell }\log\theta}\right|\le b _{ \ell }\in \ell^1(\N)
$$
and then, by the Lebesgue convergence theorem:
\begin{eqnarray}
 \label{Aequalone2}
\sum _{ \ell \in \N }b _{ \ell }A^{p _{ \ell }}=1.
\end{eqnarray}
For some $\sigma \in [0, 2\pi )$, $A=e^{i\sigma }$ and since $\sum  _{ \ell \in \N }b _{ \ell }e^{i\sigma p _{ \ell }}=1$ the calculation above implies that $\sigma \in Q,$ which contradicts the fact that $d(v_n,Q)\geq \varepsilon.$

 We deduce that
$$
\sup _{ d(v, Q)\ge \varepsilon  }\Re e\left(\sum _{ \ell\in \N }b _{ \ell }e^{iv\log \sigma  _{ \ell }}\right)<1.
$$

\end{proof}

\section*{Acknowledgments} We thank B\'en\'edicte Haas for fruitful discussions. The research of M. Doumic is supported by the ERC Starting Grant SKIPPER$^{AD}$ (number 306321). The research of M. Escobedo is supported by  Grants MTM2011-29306,  IT641-13 and SEV-2013-0323. The authors thank the referees for  enlightening and helpful remarks.

After  completion of our manuscript we have been aware of the article~\cite{BW}, where critical growth fragmentation equations are studied
using probabilistic methods.

\medskip
Received xxxx 20xx; revised xxxx 20xx.
\medskip


\begin{thebibliography}{99}




\bibitem{MR2017852} (MR2017852) [10.1007/s10097-003-0055-3] 
\newblock
J.~Bertoin.
\newblock The asymptotic behavior of fragmentation processes.
\emph{J. Eur. Math. Soc. (JEMS)}, \textbf{5} (4) (2003), 395--416.

\bibitem{MR2253162}(MR2253162) [10.1017/CBO9780511617768] 
\newblock
J.~Bertoin.
\newblock \emph{ Random Fragmentation and Coagulation Processes}, volume 102 of
  {\em Cambridge Studies in Advanced Mathematics}.
\newblock Cambridge University Press, Cambridge, 2006.


\bibitem{BW}
\newblock
   J. Bertoin and A.~R. Watson.
\newblock Probabilistic aspects of critical growth-fragmentation equations.
\newblock
preprint, \arxiv{1506.09187}.

\bibitem{BDE} (MR3162109)  [10.1088/0266-5611/30/2/025007] 
\newblock
T.~Bourgeron, M.~Doumic, and M.~Escobedo.
\newblock Estimating the division rate of the growth-fragmentation equation
  with a self-similar kernel.
\newblock \emph{Inverse Problems}, \textbf{30}  (2014), 1--28.

\bibitem{CCM} (MR2832638) [10.1016/j.matpur.2011.01.003 ] 
\newblock
M.~J. C\'aceres, J.~A. Ca\~nizo, and S.~Mischler.
\newblock Rate of convergence to an asymptotic profile for the self-similar fragmentation and growth-fragmentation equations
\newblock\emph{Journal de Math\'ematiques Pures et Appliqu\'ees}, \textbf{96}  (2011), 334--362.


\bibitem{CL2}(MR2605707) [10.1080/17513750902935208] 
\newblock
V.~Calvez, N.~Lenuzza, M.~Doumic, J.-P. Deslys, F.~Mouthon, and B.~Perthame.
\newblock Prion dynamic with size dependency - strain phenomena.
\newblock\emph{J. of Biol. Dyn.}, \textbf{4}  (2010), 28--42.

\bibitem{DG} (MR2652618) [10.1142/S021820251000443X] 
\newblock
M.~Doumic and P.~Gabriel.
\newblock Eigenelements of a general aggregation-fragmentation model.
\newblock\emph{Mathematical Models and Methods in Applied Sciences},
  \textbf{20}  (2009),  757--783.

\bibitem{Drake_1972}[10.1016/B978-0-08-016809-8.50003-6]
R.~Drake.
\newblock A general mathematical survey of the coagulation equation.
\newblock\emph{Topics in Current Aerosol Research (Part 2)} (1972),  201--376.

\bibitem{EscoMischler3} (MR2114413) [10.1016/j.anihpc.2004.06.001] 
\newblock
M.~Escobedo, S.~Mischler, and M.~Rodriguez~Ricard.
\newblock On self-similarity and stationary problem for fragmentation and
  coagulation models.
\newblock\emph{Ann. Inst. H. Poincar\'e Anal. Non Lin\'eaire}, \textbf{22}  (2005), 99--125.

\bibitem{Fredrickson_1967}[10.1016/0025-5564(67)90008-9]
A.~G. Fredrickson, D.~Ramkrishna, and H.~Tsuchiya.
\newblock Statistics and dynamics of procaryotic cell populations.
\newblock\emph{ Math Biosci.}, \textbf{1} (1967), 327--374.

\bibitem{Haas2003245} (MR1989629) [10.1016/S0304-4149(03)00045-0] 
\newblock
B.~Haas.
\newblock Loss of mass in deterministic and random fragmentations.
\newblock\emph{Stochastic Processes and their Applications}, \textbf{106} (2003), 245--277.

\bibitem{MR2650037} (MR2650037) [10.1214/09-AAP622] 
\newblock
B.~Haas.
\newblock Asymptotic behavior of solutions of the fragmentation equation with
  shattering: an approach via self-similar {M}arkov processes.
\newblock\emph{Ann. Appl. Probab.}, \textbf{20}  (2010), 382--429.

\bibitem{MR0188387} (MR0188387) [10.1007/978-3-662-29794-0] 
\newblock
E.~Hewitt and K.~Stromberg.
\newblock \emph{Real and Abstract Analysis. {A} Modern Treatment of the Theory
  of Functions of a Real Variable}.
\newblock Springer-Verlag, New York, 1965.

\bibitem{M1}(MR2250122) [10.1142/S0218202506001480] 
\newblock
P.~Michel.
\newblock Existence of a solution to the cell division eigenproblem.
\newblock \emph{Math. Models Methods Appl. Sci.}, \textbf{16}  (2006), 1125--1153.

\bibitem{Mischler:frag}[10.1016/j.anihpc.2015.01.007] 
S.~Mischler and J.~Scher.
\newblock {Spectral analysis of semigroups and growth-fragmentation equations}.
\newblock \emph{Annales de l'Institut Henri Poincare (C) Non Linear Analysis} (2015), in press.

\bibitem{DHKR2}[10.1186/1741-7007-12-17]
L.~Robert, M.~Hoffmann, N.~Krell, S.~Aymerich, J.~Robert, and M.~Doumic.
\newblock Division in escherichia coli is triggered by a size-sensing rather
  than a timing mechanism.
\newblock \emph{BMC Biology}, \textbf{12} (2014), 1--10.

\bibitem{SinkoStreifer}[10.2307/1934533]
J.~Sinko and W.~Streifer.
\newblock A new model for age-size structure of a population.
\newblock \emph{Ecology}, \textbf{48} (6) 1967, 910--918.

\end{thebibliography}
\end{document}